\documentclass[reqno,12pt]{amsart}

\usepackage{amsmath}
\usepackage{amssymb}
\usepackage{amsfonts}
\usepackage{mathtools}
\usepackage{enumerate}
\usepackage{comment}
\allowdisplaybreaks

\usepackage{palatino}
\usepackage[T1]{fontenc}
\usepackage[mathscr]{eucal}
\usepackage{dsfont}

\usepackage{fullpage}

\usepackage{acronym}
\usepackage{latexsym}
\usepackage{paralist}
\usepackage{wasysym}
\usepackage{xspace}

\usepackage[dvipsnames,svgnames]{xcolor}
\colorlet{MyBlue}{DodgerBlue!60!Black}
\colorlet{MyGreen}{DarkGreen!85!Black}

\usepackage[font=small,labelfont=bf]{caption}
\usepackage{subfigure}
\usepackage{tikz}
\usetikzlibrary{calc,patterns}

\usepackage[authoryear,comma]{natbib}

\usepackage{hyperref}
\hypersetup{
colorlinks=true,
linktocpage=true,
pdfstartview=FitH,
breaklinks=true,
pdfpagemode=UseNone,
pageanchor=true,
pdfpagemode=UseOutlines,
plainpages=false,
bookmarksnumbered,
bookmarksopen=false,
bookmarksopenlevel=1,
hypertexnames=true,
pdfhighlight=/O,
urlcolor=MyBlue!60!black,linkcolor=MyBlue!70!black,citecolor=DarkGreen!70!black, 
pdftitle={},
pdfauthor={},
pdfsubject={},
pdfkeywords={},
pdfcreator={pdfLaTeX},
pdfproducer={LaTeX with hyperref}
}

\numberwithin{equation}{section}  
\usepackage[sort&compress,capitalize,nameinlink]{cleveref}
\crefname{app}{Appendix}{Appendices}

\crefrangeformat{equation}{\upshape(#3#1#4)\textendash(#5#2#6)}

\usepackage[textwidth=30mm]{todonotes}

\usepackage{soul}
\setstcolor{red}
\sethlcolor{SkyBlue}

\usepackage{enumitem}

\newcommand{\debug}[1]{{\color{black}#1}}

\theoremstyle{plain}
\newtheorem{theorem}{Theorem}
\newtheorem{corollary}[theorem]{Corollary}
\newtheorem*{corollary*}{Corollary}
\newtheorem{lemma}[theorem]{Lemma}
\newtheorem{proposition}[theorem]{Proposition}

\theoremstyle{definition}

\newtheorem*{definition*}{Definition}
\newtheorem{assumption}[theorem]{Assumption}

\newtheorem*{hypothesis*}{Hypothesis}

\theoremstyle{remark}
\newtheorem{remark}[theorem]{Remark}
\newtheorem*{remark*}{Remark}
\newtheorem*{notation*}{Notational remark}

\numberwithin{theorem}{section}

\DeclarePairedDelimiter{\floor}{\lfloor}{\rfloor}

\DeclarePairedDelimiterX{\braket}[2]{\langle}{\rangle}{#1,#2}

\DeclarePairedDelimiterX{\inner}[2]{\langle}{\rangle}{#1,#2}
\DeclarePairedDelimiterX{\setdef}[2]{\{}{\}}{#1:#2}

\DeclarePairedDelimiterXPP{\probof}[1]{\Prob}{(}{)}{}{%

#1}

\DeclarePairedDelimiterXPP{\exof}[1]{\Expect}{[}{]}{}{%

#1}

\newcommand{\Expect}{\mathbf{\debug E}}
\newcommand{\Prob}{\mathbf{\debug P}}

\DeclareMathOperator{\supp}{\debug{supp}}

\newcommand{\ind}{\mathds{1}}

\newcommand{\cC}{\ensuremath{\mathcal C}}

\newcommand{\cF}{\ensuremath{\mathcal F}}

\newcommand{\E}{\ensuremath{\mathbb{E}}}

\newcommand{\N}{\ensuremath{\mathbb{N}}}

\renewcommand{\P}{\ensuremath{\mathbb{P}}}

\newacro{NE}{Nash equilibrium}
\newacroplural{NE}[NE]{Nash equilibria}
\newacro{PNE}{pure Nash equilibrium}
\newacroplural{PNE}[PNE]{pure Nash equilibria}
\newacro{MNE}{mixed Nash equilibrium}
\newacroplural{MNE}[MNE]{mixed Nash equilibria}
\newacro{PFNE}{prior-free Nash equilibrium}
\newacroplural{PFNE}[PFNE]{prior-free Nash equilibria}
\newacro{WE}{Wardrop equilibrium}
\newacroplural{WE}[WE]{Wardrop equilibria}
\newacro{SO}{socially optimum}
\newacro{SU}{social utility}
\newacro{BEq}{best equilibrium}
\newacro{WEq}{worst equilibrium}
\newacro{KKT}{Karush\textendash Kuhn\textendash Tucker}
\newacro{OD}[O/D]{origin-destination}
\newacro{PoA}{price of anarchy}
\newacro{PoS}{price of stability}
\newacro{PoCS}{price of correlated stability}
\newacro{BPR}{bureau of public roads}
\newacro{FIP}{finite improvement property}
\newacro{CLT}{central limit theorem}

\newacro{BPG}{buck-passing game}
\newacro{SBPG}{stochastic buck-passing game}
\newacro{MBPG}{mixed extension of the buck-passing game}

\newcommand*{\figref}[2][]{%
  \hyperref[{fig:#2}]{%
    Figure~\ref*{fig:#2}%
    \ifx\\#1\\%
    \else
      \,#1%
    \fi
  }%
}

\begin{document}
\title {Evolution of discordant edges in the voter model on random sparse digraphs}
%

 \author[F.~Capannoli]{Federico Capannoli$^{}$}
 \address{$^{}$Mathematical Institute, Leiden University, Gorlaeus Gebouw, BW-vleugel, Einsteinweg 55, 2333 CC Leiden, The Netherlands}

\email{f.capannoli@math.leidenuniv.it}


\maketitle          

\begin{abstract}

We explore the voter model dynamics on a directed random graph model ensemble (digraphs), given by the Directed Configuration Model. The voter model captures the evolution of opinions over time on a graph where each vertex represents an individual holding a binary opinion. Our primary interest lies in the density of discordant edges, defined as the fraction of edges connecting vertices with different opinions, and its asymptotic behavior as the graph size grows to infinity. This analysis provides valuable insights, not only into the consensus time behavior but also into how the process approaches this absorption time on shorter time scales.
Our analysis is based on the study of certain annealed random walk processes evolving on out-directed, marked Galton-Watson trees, which describe the locally tree-like nature of the considered random graph model. Additionally, we employ innovative coupling techniques that exploit the classical stochastic dual process of coalescing random walks. We extend existing results on random regular graphs to the more general setting of heterogeneous and directed configurations, highlighting the role of graph topology in the opinion dynamics.

\end{abstract}

\section{Introduction}

In recent times, new tools have been developed that make it possible to analyze interacting particle systems (IPSs) evolving on large finite graphs, enriching the pre-existing infinite volume literature, see \cite{Lig85}, developed since the late 70s. In this finite volume setting, the voter model is one of the most studied IPSs, mainly due to its stochastic dual process that translates its analysis into questions about coalescing random walks. The classical two-opinions voter model was introduced in \cite{CS73}, then it was further studied in \cite{HL75} on the lattice and in \cite{Cox89} on the torus, and it represents one of the simplest opinion dynamics evolution that can be modeled via a Markov process. As many other IPSs, the voter model can be seen as a simplified model to understand some behaviours of social or real life networks. Therefore, it became natural to generalise it on random graphs. The voter model and coalescing random walks on finite graphs, mainly studied via the consensus and coalescence times, were recently investigated on random graphs, see e.g. \cite{CFR09}, \cite{FO22}, \cite{HSDL22}, \cite{ACHQ23} and \cite{vBK24}. We are interested in a directed random graph model, called directed configuration model (DCM). The DCM is a random graph model in which every vertex have a prescribed in- and out-degree and the randomness is specified by the edge set. It resembles the directed version of the configuration model introduced in \cite{Bol80}, while the DCM was first studied in \cite{CF04}.  Recently there has been an increasing interest in studying the topological properties of this directed ensemble (see e.g. \cite{vdHOC18}, \cite{CP21}, \cite{vdHP24}) as well as the evolution of stochastic processes over such random geometry (see e.g. \cite{BCS18}, \cite{CQ20}, \cite{CQ21a}, \cite{CCPQ21}, \cite{QS23}.

We can informally describe the voter dynamics as follows. On a given locally finite graph each vertex, individual, has an initial mark, opinion, usually denoted by 0 or 1. After waiting a random amount of time given by a collection of rate-one Poisson clocks, an individual changes its opinion by adopting the one of a randomly chosen neighbour. Naturally in this setting, the main object of interest is the distribution of the so-called \emph{consensus time}, i.e., the first time at which all the vertices share the same opinion. Voter model on DCM was first studied by \cite{ACHQ23} for a general class of in- and out-degree sequences. The authors showed the precise first-order asymptotic of the consensus time after a proper scaling. In particular the expected consensus time scales linearly in the size of the graph, with an explicit pre-constant that depends on the heterogeneity of the degrees.

In this paper we conduct a more detailed study on the voter model on DCM, which would simultaneously provide all the information about the consensus time and the process by which it was reached. In particular, we examine the evolution of the density of \emph{discordant edges}, that is, the ratio of edges that have different opinion. It is an interesting object to study as it resembles a richer observable w.r.t. the consensus time, and describes exactly the perimeter of the set of vertices with one of the two opinions. A crucial work in the rigours literature that depicts a complete analysis of the discordant edges behaviour is \cite{ABHHQ22} on the regular random graphs. There, the authors show that the expected fraction of discordant edges has a interesting interplay with the geometry of the graph. In particular, the choice of the time scale with respect to the size of the network will show drastically different behaviours of the process. Such work although was restricted to the, symmetric (undirected) and degree homogeneous, regular case.

\subsection*{Our contribution} Adapting the basic approach of \cite{ACHQ23} in our setting, we show that it is possible to get the explicit behaviour of the expected fraction of discordant edges on the sparse DCM. More precisely, we are able to prove that a similar behaviour as seen in the random regular case is preserved, i.e. the process first drops to a constant plateau, then it stabilizes in this metastable state for a long time and then finally, when the time scale is of the same order of the consensus time, it approaches zero. On the other hand some interesting outcomes arose from our analysis. Unlike the random regular case, we observe a different behaviour regarding the explicit function leading the first-order asymptotic for short and moderate time scales. This is due to the directed nature and the inhomogeneity of the underlying geometry. The new explicit pre-constant turns out to be an uniformly bounded function of the degree sequence that depends only on the average degree and a spectral quantity that governs the homogeneity of the in- and out-degrees. Moreover, it is worth to point out that this work gives a contribution to the literature of IPSs on random directed graphs, that is still far from being completely understood. As a consequence of duality, we will see that the proof depends on studying properties of random walks on random environment. We look at joint law of the process and the graph dynamics together. Such analysis was possible thanks to innovative annealing techniques used in \cite{BCS18}, \cite{CCPQ21} and \cite[Section 6.1]{ACHQ23}.

\subsection*{Outline} In Section \ref{sec: Models and results} we define formally the voter model together with the directed configuration model random graph, and eventually we state our main result. In Section \ref{sec:Preliminaries} we will introduce the consensus time for the voter model and how it relates to the discordant edges, together with a crucial tool for our analysis, given by the dual system of coalescing random walks. Then we briefly describe the directed random environment of interest, showing that its local geometry can be well-approximated by a Galton-Watson tree. Finally, we define the observable related to the random walks evolving on the DCM and its relation to the so-called annealed random walk. In Section \ref{sec: Proof of the main result} we give a complete proof of the main result. First we prove the result for short time scales, and afterwords we extend it to any time scales.

\section{Models and results} \label{sec: Models and results}

In this section we formally introduce the two model of interest for this paper: the voter model on directed graphs and the directed configuration model. After that we state our main result, describing the asymptotic behaviour of the expected fraction of discordant edges on a typical realisation of the random environment.

\subsection{Voter model}\label{sec:voter model}
Given a directed, strongly connected graph $G=(V,E)$, we define the voter model on $G$ as the continuous-time Markov process $(\eta_t)_{t\geq 0}$ with state space $\{0,1\}^V$ and infinitesimal generator $\mathcal{L}_{vm}$ as
\begin{equation*}
    (\mathcal{L}_{vm}f)(\eta) = \sum_{x\in V} \sum_{\substack{y\in V : \\ (x,y)\in E}} \frac{1}{d_x^+} \big(f(\eta^{x\to y}) - f(\eta)\big), \quad f:\{0,1\}^V \to \mathbb{R}\,,
\end{equation*}
where $(x,y)\in E$ denotes a directed edge exiting $x$ and entering $y$, $d_x^+ = |\{z\in V \mid (x,z)\in E\}|$ 
is the out-degree of $x$ and
\begin{equation*}
    \eta^{x\to y} (z) = 
    \begin{cases}
        \eta(y),& \quad \text{if } z=x\,,\\
        \eta(x),& \quad \text{otherwise}\,. \\
    \end{cases}
\end{equation*}
For any $u\in[0,1]$, let $\mathbf{P}_u$ be the law of the voter model $(\eta_t)_{t\geq 0}$ with initial distribution $\eta_0 = \text{Bern}(u)^{\otimes V}$, and $\mathbf{E}_u$ its expectation.
Sometimes we may adopt the equivalent notation $x\to y$ in place of $(x,y)$ to emphasise the fact that the edge is directed from $x$ to $y$. For any $x\in V$ and $t\in \mathbb{R}^+$, $\eta_t(x)$ represents the state of node $x$ at time $t$ in terms of the binary state $\{0,1\}$, to be interpreted as the opinion of the individual $x$ at time $t$. In other words, the process captures the evolution of the opinion dynamics starting from the initial configuration of opinions given by $\eta_0 = \{\eta_0(x) \mid x\in V\}$. The Markovian evolution defined by the generator $\mathcal{L}$ can be described as follows. Give to each directed edge $(x,y)$ an exponential clock of rate $1/d_x^+$. When the clock associated to an edge $x\to y$ rings, vertex $x$ adopts the opinion of vertex $y$. Similarly to other interacting particle systems, such description gives rise to the so-called \emph{graphical representation} for the voter model. We refer to \cite{Lig85} and \cite{Lig99} for all the details concerning the matter.
Notice that such Markov process has two absorbing states, corresponding to the monochromatic configurations $\bar{0}$ and $\bar{1}$ consisting of all $0$'s and $1$'s, respectively. If we assume $G$ to be finite, then it can be shown that almost surely the process will reach one of the two absorbing states in finite time. This setting naturally leads to the question of determining the time such that the system reaches the absorbing states, called \emph{consensus time}, and defined as
\begin{equation}
    \tau_{\rm cons} = \inf\{t\geq0 : \eta_t \in \{\bar{0}, \bar{1}\}\}\,.
\end{equation}
In the literature this hitting time was deeply studied in a wide variety of underlying random and non-random, finite and infinite volume geometries. Knowing the behaviour of the consensus time is generally a difficult task that strongly depends on the graph structure. Additionally, it is not very informative regarding the evolution of the process, as it lacks information about how the opinion dynamics led to such a consensus. It is possible to perform a different, more detailed study on the voter model from which, at the same time, we can derive all the information regarding the consensus time and how did the process reach it. In the present paper we analyse the evolution of the density of \emph{discordant edges}. More precisely, let us denote the set of discordant edges at time $t$ as
$$
D_t = D_t^{(n)} = \{e=(x,y)\in E \;:\; \eta_t(x)\neq \eta_t(y)\}\,.
$$
Therefore, we set the density of discordant edges at time $t$ to be
\begin{equation} \label{eq:discordant density}
   \mathcal{D}_t =\mathcal{D}_t^{(n)} = \frac{|D_t|}{|E|}\,. 
\end{equation}
The aim of this paper is to study the asymptotic evolution of the latter quantity, as the size of the underlying graph grows to infinity. The evolution of (the fraction of) discordant edges has a proper interest as it exactly captures the way in which the two opinion compete before reaching consensus. Furthermore, as shown in \cite{CCC16}, there is an interesting interplay between the Fisher-Wright diffusion, seen as scaling limit of the fraction of, say, blue opinions, and the scaling limit of the fraction of discordant edges \eqref{eq:discordant density}.

\subsection{Directed configuration model} \label{sec:DCM}
For any $n\in\mathbb{N}$, let $[n]:= \{1,\dots,n\}$ be a set of $n$ labeled nodes. For any vertex $x\in[n]$, let $d_x^+$ (resp. $d_x^-$) be its out-degree (resp. in-degree), that is the number of vertices that are connected to $x$ via a directed edge that is exiting (resp. entering) $x$. Define $\mathbf{d}_n = ((d_1^-,d_1^+),\dots,(d_n^-,d_n^+))$ to be a deterministic bi-degree sequence with the following constraint
\begin{equation} \label{graphical}
    m = m_n := \sum_{x\in[n]}d_x^- = \sum_{x\in[n]}d_x^+\,.
\end{equation}
The randomness of the model comes from the mechanism in which the edges are formed. This is a result of the following uniform pairing procedure. Assign to each vertex $x\in[n]$, $d_x^+$ labeled \emph{heads} and $d_x^-$ labeled \emph{tails}, denoting the in- and out-stubs of $x$ respectively. At each step, select a tail $e$ that was not matched in a previous step, and a uniform at random head $f$ among the unmatched ones, then match them and add the directed edge $ef$ between the vertex incident to $e$ and the one incident to $f$ to the edge set $E$. Continue until there are no more unmatched heads and tails. Note that the constraint in \eqref{graphical} ensures that such uniform matching ends without any stub left unmatched. This random procedure gives rise to a so-called configuration, and it uniquely determines the corresponding random, directed graph $G=G_n=([n],E)$. We say that a graph $G_n$ is sampled from the Directed Configuration Model DCM $=$ DCM($\mathbf{d}_n$) with a given degree sequence $\mathbf{d}_n$, if it is sampled according to the procedure above. We denote by $\P$ the law of $G$. We will be interested in studying the asymptotic regime in which $n\to\infty$, and we will say that $G$ has a certain property $E_n$ \emph{with high probability (w.h.p.)}, if
$$ 
\P(E_n) \to 1\,,  \text{ as } n\to\infty\,.
$$
Let $d^\pm_{\rm max} = \max_{x\in[n]}d_x^\pm$ and $d^\pm_{\rm min} = \min_{x\in[n]}d_x^\pm$.  We will consider the following assumptions over $\mathbf{d}_n$:
\begin{assumption} \label{deg-assumptionn} There exist some constants $C,C' \geq 2$ such that for any $n\in \mathbb{N}$
\begin{center}
    \begin{align*}
        (1)& \quad d_{\rm min}^\pm \geq 2\,, \\
        (2)& \quad d_{\rm max}^+ \leq C\,, \\
        (3)& \quad d_{\rm max}^- \leq C'\,.
    \end{align*}  
\end{center}
\end{assumption}
Under the assumption (1) it holds that the resulting graph realisation will be strongly connected with high probability (\cite{CF04}), while assumptions (2) and (3) guarantee that the graph is \emph{sparse}, in the sense that the number of edges grows at most linearly in the size of the graph, i.e. $m = \mathcal{O}(n)$.

\subsection{Main result}
Before stating the main result, we introduce some relevant functions of the degree sequence of the DCM.
Let
\begin{equation}\label{eq:def-rho-gamma-delta-beta}
	\begin{split}
\delta&=\delta_n\coloneqq\frac{m}{n}\,,\qquad\quad\quad\quad\:\:\:\beta=\beta_n\coloneqq\frac1{m}\sum_{x\in[n]}(d_x^-)^2\,, \\
\rho&=\rho_n\coloneqq\frac{1}{m}\sum_{x\in[n]}\frac{d_x^-}{d^+_x}\,, \quad\quad\gamma=\gamma_n\coloneqq\frac{1}{m}\sum_{x\in[n]}\,\frac{(d^-_x)^2}{d^+_x}\,,
 \end{split}
\end{equation}
where $m$ as in \eqref{graphical}. Notice that, under Assumptions \ref{deg-assumptionn}, all the above quantities are $\Theta(1)$ and bounded away from zero. Moreover, define
\begin{equation}\label{theta}
    \vartheta=\vartheta_n( \mathbf{d}^+, \mathbf{d}^-)\coloneqq \frac{\delta}{\frac{\gamma-\rho}{1-\rho}\, \frac{1-\sqrt{1-\rho}}{\rho}+\beta-1}\,.
\end{equation}

The following represents the main contribution of this paper. It gives a complete picture of the asymptotic behaviour of the expected fraction of discordant edges for the voter model with high probability with respect to the law of the DCM.

\begin{theorem} \label{thm:main}
    Suppose that the degree sequence satisfies Assumption \ref{deg-assumptionn}. Fix $u\in(0,1)$ and let $n\in\mathbb{N}$. Consider the voter model on the directed configuration model $G_n=([n],E)$ with initial distribution $\text{Bern}(u)^{\otimes [n]}$. Then, for any non-negative sequence $t_n$ such that $\lim_{n\to \infty}t_n$ and $\lim_{n\to \infty}t_n/n$ exist,

it holds that 
\begin{equation} \label{eq:main}
\bigg|\mathbf{E}_u[\mathcal{D}_{t_n}] - 2u(1-u)\,\varphi(t_n)\,e^{-2\frac{t_n}{n}\vartheta^{-1}}\bigg| \overset{\P}{\longrightarrow}0\,,
\end{equation}
where 
\begin{equation}\label{eq:phi}
\varphi(t) = 1 - \frac{1}{2\,\delta} \sum_{k\geq 0} e^{-2t} \frac{(2t)^k}{k!}\Bigg(\sum_{s=1}^{\floor{\frac{k-1}{2}}} 2^{-2s}\,C_s\,\rho^s\,\mathds{1}_{k>2} +\mathds{1}_{k>0}\Bigg),\quad   t\geq 0,    
\end{equation}
$\vartheta$ as in \eqref{theta}, $\rho$ and $\delta$ as in \eqref{eq:def-rho-gamma-delta-beta},  and $C_s$ denote the Catalan numbers, i.e.
$$
\quad C_s = \frac{1}{s+1}\binom{2s}{s}\,.
$$
\end{theorem}

\begin{remark}
We can immediately check that the series in \eqref{eq:phi} is converging uniformly in $t$ to

\begin{equation}\label{eq:limiting value phi}
    \varphi(\infty)= \varphi_n(\infty) =1-\frac{1-\sqrt{1-\rho}}{\delta\,\rho}\,.
\end{equation}
Furthermore, as it will be shown in the proof of Proposition \ref{prop:first moment short time scales}, if $t=t_n$ is a diverging sequence then $\varphi(t)$ is close to $\varphi(\infty)$ as $n\to \infty$.
\end{remark}
Similarly to what the authors proved in \cite{ABHHQ22}, from the expression in \eqref{eq:main} we observe that there are four different time scales for the evolution of the voter model that show different behaviours of the expected density of discordant edges.
\begin{enumerate}
    \item \textit{Short time scale}. If $t_n$ is of order one, i.e. $t_n=t=\Theta(1)$, then the exponential factor doesn't play any role and the leading term is given by the function $2u(1-u)\varphi(t)$. This is related to the event that two coalescing annealed walks evolving on a out-directed, multi-type Galton-Watson tree do not meet and chase each other on the same branch of the tree within time $t$.
    \item \textit{Intermediate time scale}. If $t_n$ diverges slowly, i.e. it is such that $\lim_{n\to \infty} t_n = \infty$ and $t_n = o(n)$, then as $n\to\infty$ the density of discordances stabilises around the limiting value $2u(1-u)\varphi(\infty)$ as in \eqref{eq:limiting value phi}. The choice of such a range for intermediate time scales is due to Theorem \ref{thm:meeting from stationarity}, which indicates that the consensus time for the voter model in the sparse DCM is linear in $n$.
    \item \textit{Long time scale}. If $t_n$ is of the consensus time order, i.e. $t_n = \ell\,n$ with $\ell\in (0,\infty)$, then the voter model is approaching consensus and the exponential factor in \eqref{eq:main} becomes relevant. In terms of the discordant edges behaviour this reflects into a drastic descent from the previous plateau, approaching zero. See also Figure \ref{fig:simulation discordant edges}.
    \item \textit{Consensus}. If $t_n$ exceeds the consensus time scale, i.e. $\lim_{n\to\infty} t_n/n =\infty$, then the expression in \eqref{eq:main} vanishes.
\end{enumerate}

\begin{figure}
    \centering
    \includegraphics[scale=0.5]{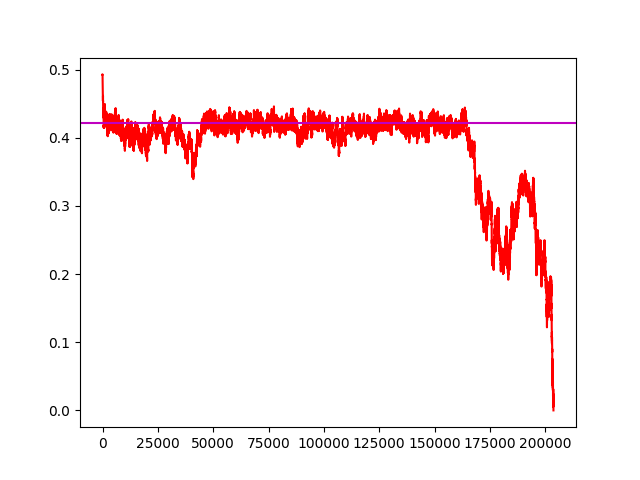}
    \caption{A single simulation of the voter model with $u=1/2$ on a quenched realization of DCM with $n=1000$ vertices and (in/out) degrees ranging between 2 and 5. In red, the fraction of discordant edges; in magenta, the constant line having value $2u(1-u)\varphi(\infty)$. As predicted, the phenomenology consists of an initial drop to the magenta line, then, for the intermediate time scales, a metastable behaviour stabilized around $2u(1-u)\varphi(\infty)$, then finally a rapid drop to zero on the consensus scale.}
    \label{fig:simulation discordant edges}
\end{figure}

\begin{remark}[Variability of degree sequence]
 Recall first that the authors in \cite{ACHQ23} showed that from the expression $\vartheta$ in \eqref{theta} one could retrieve relevant information about how fast does the consensus time happen depending on the regularity/variability of the degree sequence. We can get a similar information on some time scales exploiting such a fact.
 \begin{itemize}
    \item Consider the special cases in which the graph is \emph{out-regular}, i.e. for a fixed $d\geq 2$ the degree sequence $\mathbf{d}$ is such that $d^+_x\equiv d$ for all $x\in [n]$, and the one in which it is \emph{regular}, i.e. $d^+_x=d_x^-\equiv d$ for all $x\in [n]$ . Note that for both cases $\delta = d$ and $\rho =1/d$, thus $\varphi(\infty)$ has the same expression given by
    \begin{equation*}
    \sqrt{\frac{d}{d-1}}\,.
    \end{equation*}
    Such value matches exactly the value of of $\vartheta$ for the \emph{$d$-regular case}. We would like to point out that such result, only for the simplified $d$-regular random directed graphs setting, could also be retrieved from \cite{CCC16}. For general out-regular random graphs, analyzing the quantity $\vartheta$, in \cite{ACHQ23} it has been showed that the expected consensus time happens faster than the regular case, depending on the size of the second moment of the in-degree sequence. We can get from a different perspective a new, coherent interpretation of the speed at which the consensus phenomena is happening.

    \item If we consider long time scales, that is such that $\lim_{n\to \infty} t_n/n=\ell \in(0,\infty)$, by Theorem \ref{thm:main} we have that the expected density of discordant edges is close to 
    $$ 2u(1-u)\,\varphi(\infty)\,e^{-2\ell \vartheta^{-1}}\,. $$ 
    By the previous observations, we conclude that, for long time scales, the expected density of discordant edges for $n$ large is smaller in the out-regular case with respect to the regular one. The difference becomes larger as the variability of the in-degrees increases, while they coincide only when it approaches zero, that is in the regular case.

    \item Finally we can make the following observations on the values of $\varphi(\infty)$. Once we fix the average degree $\delta$, we can observe that the function $x\mapsto 1-\frac{1-\sqrt{1-x}}{\delta\,x}$ is non-increasing in its domain $(0,1/2)$, with a small range. This implies that, at least on short and moderate time scales, the fraction of discordances remains larger the smallest the value that $\rho$ can attain. This happens when there are a lot of vertices with small in-degree and large out-degree. The same conclusion does not hold on the consensus scale as one has to compare such value with the exponential factor containing $\vartheta$.
\end{itemize}

\end{remark}

To the best of our knowledge, prior to our contribution in the current literature, the only work that adequately addressed the problem is \cite{ABHHQ22}, in the random regular graphs setting. The only other known ensemble is the complete graph, where the number of discordant edges is trivially the product of the numbers of vertices holding the two respective opinions. For a comparison with our result, we state the result for the random regular graph setting.

\begin{theorem}[{Theorem 1.1, \cite{ABHHQ22}}] \label{thm:discordant edges random regular graphs}
Fix $d\geq 3$, $u\in(0,1)$ and let $\theta_d = \frac{d-1}{d-2}$. For $n\geq d+1$ consider the voter model on a $d-$ regular random graph $G_{d,n}$ with $n$ vertices and initial distribution $\text{Bern}(u)^{\otimes V}$. Then, for any non-negative sequence $(t_n)_{n\in\mathbb{N}}$ such that the limit of $t_n$ and $t_n/n$ exists, it holds that
\begin{equation} \label{eq:expected discordant edges random regular graphs}
    \bigg|\mathbf{E}_u[\mathcal{D}_{t_n}] - 2u(1-u)\,f_d(t_n)\,e^{-2\frac{t_n}{n}\theta^{-1}_d}\bigg| \overset{\P}{\longrightarrow}0\,, \quad \text{as } n\to \infty\,,
\end{equation}
where $\P$ is the law of the random regular graph and $f_d:\mathbb{R}_+\to [0,1]$ is an explicit function such that 
$$f_d(t_n) \to \theta_d\,,$$
as $n\to \infty$.
\end{theorem}


The expression in \eqref{eq:expected discordant edges random regular graphs} shows that, similarly to our result, different time scales produce different behaviours. Moreover, the function $f_d$ that governs the leading term of the evolution up to linearity, replaced by $\varphi$ in our setting, is related to the first meeting time of two walks on an infinite, deterministic $d$-regular tree. This has to do with the locally-tree-like nature of the sparse $d$-regular random graph. We will exploit such beautiful property also in our directed, heterogeneous setting, and, as expected, this will lead to a comparable result. All the details are discussed in Section \ref{sec:Local structure of the graph}.    



\section{Preliminaries} \label{sec:Preliminaries}

\subsection{Local structure of the graph}\label{sec:Local structure of the graph} A consequence of the fact that we are considering realisations of sparse random graphs is that their local structure is locally tree-like. This is related to the fact that, as many others sparse random graphs models, the local weak limit of the sparse DCM is a Galton-Watson tree. See \cite{vdH16} for a modern introduction to the topic. 

We want to compare the exploration process of a neighborhood of $G$ with an exploration process of a marked Galton-Watson tree. For any fixed $x\in[n]$ and any $h=h_n>0$, define $\mathcal{B}^+_x(h)$, the $h$-out-neighborhood of vertex $x$, to be the set of paths starting from $x$ of length at most $h$. We generate $\mathcal{B}^+_x(h)$ using the breadth-first procedure (BF) starting from $x$ as priority rule, iterating the following steps:
\begin{enumerate}
    \item \label{item: 1}pick the first available unmatched tail $e$ according to BF starting from $x$;
    \item \label{item: 2}pick uniformly at random an \emph{unmatched} head $f$;
    \item \label{item: 3} draw the resulting directed edge $ef$. Continue until the graph distance from an unmatched tail in Item \ref{item: 1} to $x$ exceeds $h$.
\end{enumerate}
Let $x\in[n]$, and define a marked (out-directed) random tree $\mathcal{T}^+_x$ rooted at $x$ as follows: the root is assigned mark $x$, and all other vertices an independent mark $\ell\in[n]$ with probability $\frac{d_\ell^-}{m}$. Each vertex with mark $\ell\in[n]$ has $d_\ell^+$ children. Note that $\mathcal{T}^+_x$ is obtained by gluing together $d_x^+$ independent Galton-Watson trees with offspring distribution
\begin{equation}
    \label{eq: size-biased offspring distribution}
   \mu^+(k) \coloneqq \sum_{x\in[n]} \frac{d_x^-}{m} \ind(d^+_x = k)\ , \qquad k\in\N\ .
\end{equation}
Let $\mathcal{T}^+_v (h)$ be a subtree of $\mathcal{T}^+_v$ given by its truncation up to generation $h$. A classical description of the coupling between $\mathcal{B}^+_v(h)$ and $\mathcal{T}^+_v (h)$ can be found e.g. in \cite[Sec. 2.2]{CCPQ21} and \cite[Sec. 4.1]{ACHQ23}.

 \begin{lemma}[{Lemma 4.1, \cite{ACHQ23}}]\label{lemma:LTL-structure}
Assume the degree sequence satisfy Assumption \ref{deg-assumptionn}. Let $v\in[n]$, then for any $h>0$ there exists a coupling between $\mathcal{B}^+_v(h)$ and $\mathcal{T}^+_v (h)$ having law $\hat{\P}$ such that
\begin{equation*}
     \hat{\P}\left(\mathcal{B}^+_v(\hbar) \neq  \mathcal{T}^+_v (\hbar)\right) =o(1)\,, 
\end{equation*}
where
\begin{equation}\label{eq:def-h-bar}
\hbar =\hbar_n\coloneqq \frac{\log(n)}{5\,\log(d_{\rm max}^+)}\,.
\end{equation}
\end{lemma}

\subsection{Coalescing random walks}
The consensus time for the voter model is related to an observable of a different Markovian process called \emph{coalescing random walks} (CRWs). In fact, this process can be interpreted as the stochastic \emph{dual} process of the voter model, in the sense that it tracks back in time the origin of the opinions. More precisely, let $\{(X_t^x)_{t\geq 0}\}_{x\in V}$ be a collection of rate one continuous-time random walk such that $X_0^x = x$, for any $x\in V$. This system of random walks is such that each time two walks meet at the same vertex they collapse (coalesce) into a new, independent single walk. Therefore, similarly to the consensus time, it is well-defined the \emph{coalescing time} $\tau_{\rm coal}$, the first time such that all the RWs coalesced into a single one. Note that, under the assumption that $G$ is finite, any two walks meet in finite time, thus $\tau_{\rm coal} < \infty$ almost surely. In full generality, it holds that $\tau_{\rm cons} \leq \tau_{\rm coal}$ almost surely, as the CRWs system is the dual process of the voter model. Under some further assumption on $G$ it can be proved that, in some cases, consensus and coalescing times are not that far apart. In particular, in \cite{Oli13} the author proved that under some \emph{mean field conditions} the asymptotic behaviour, with respect to the size of the network growing to infinity, of the consensus and coalescing time can be well-approximated by the expected meeting time of two independent random walk starting from stationarity.

\subsection{Random walks on sparse DCM} \label{sec: Random walks on sparse DCM}
We define $(X_t)_{t\geq 0}$ to be a continuous-time rate one simple random walk on a directed graph $G=([n], E)$ evolving on the \emph{out-degrees}, i.e. as the Markov chain with state space $[n]$ and generator
\begin{equation}
    (\mathcal{L}_{rw}f)(x) = \sum_{y\in[n]} \frac{|\{e\in E : e = (x,y)\}|}{d_x^+} \,[f(y)-f(x)], \quad f:[n]\to \mathbb{R}\,.
\end{equation}
Let $\mathcal{P}([n])$ be the set of probability measures on the vertex set $[n]$. Using a slight abuse of notation with respect to the law of the voter model, we denote by $\mathbf{P}_\mu$ the law of $(X_t)_{t\geq0}$ with $X_0 \sim \mu$ and $\mu \in \mathcal{P}([n])$, and with $\mathbf{E}_\mu$ its expectation. We adopt the usual notation $\mathbf{P}_x$, $\mathbf{E}_x$ whenever $\mu$ is a Dirac mass at some $x\in [n]$. In our setting we will be considering random walks evolving on the random environment depicted by the DCM. As a consequence, the latter probability measure describing the random walk's law will become random measures according to $\P$, the law of the environment.

As it will become clear in Section \ref{sec: Proof of the main result}, the main observable of interest of this paper is the first meeting time of two copies of independent random walks, defined as
\begin{equation} \label{eq:tau meet}
    \tau_{\rm meet }^{(x,y)} = \inf\{t\geq 0 : X_t^x = X_t^y \}\,,
\end{equation}
where $(X_t^z)_{t\geq 0}$, $z\in[n]$, represents a random walk having initial position $X_0^z = z$. More generally, when the initial distribution of the walks is given by the realisation of two measures $\mu$ and $\nu$, we write $ \tau_{\rm meet }^{\mu\otimes \nu}$ to describe the first meeting time of two independent walks $X,Y$ such that $X_0\sim \mu$ and $Y_0\sim \nu$.

We will prove our main result through continuous-time random walks due to the dual relation with the voter model. Nevertheless, it will be useful to work with the skeleton chain given by the discrete-time, asynchronous version of the process defined by the transition matrix
\begin{equation}
    \mathrm{P}(x,y) =  \frac{|\{e\in E : e = (x,y)\}|}{d_x^+}\,, \quad x,y\in[n]\,.
\end{equation}
We will denote by $P_\mu$ and $E_\mu$ the law and expectation of such Markov chain whenever the initial position of the walk is distributed according to $\mu$, where $\mu \in \mathcal{P}([n])$. 

In the following we report a result that was crucial in order to proceed with our discordant edges analysis. It gives insights about the asymptotic behaviour of the first meeting time of two random walks starting from their stationary distribution $\pi$. In particular, it shows that for $n=|V|$ large enough, with high probability the distribution and the expectation of the meeting time, rescaled by an explicit linear factor of the size of the network, are well-approximated by the ones of a rate one exponential random variable.

\begin{theorem}[{Theorem 3.1, \cite{ACHQ23}}] \label{thm:meeting from stationarity}
Let $(\mathbf{d}^+,\mathbf{d}^-)$ satisfy Assumption~\ref{deg-assumptionn} and $G$ be sampled from DCM$(\mathbf{d}^+,\mathbf{d}^-)$. Then, letting $\tau^{\pi\otimes\pi}_{\mathrm{meet}}$ denote the first meeting time of two independent stationary random walks, it holds
\begin{equation}\label{eq:meeting time law}
        d_W\left(\frac{\tau^{\pi\otimes\pi}_{\mathrm{meet}}}{\frac12\ \vartheta \times n},\ {\rm Exp}(1)\right)\overset{\P}\longrightarrow 0\,,
        \end{equation}
    with $d_W (\cdot,\cdot)$ denoting the Wasserstein 1-distance.

   \end{theorem}

\subsection{Annealed random walks}
We conclude this section explaining briefly one of the main tools that we used to derive our results. Since we will be studying random walks evolving on random environments, the laws of the walks will be a random variable with respect to the law of the environment, thus depending on the realisation of the graph $G$. We are interested in stating asymptotic results in probability w.r.t. the law $\P$ of the random graph, and to this aim we will often be interested in computing expectations $\E$ of the laws the observables of interest; in other words, we want to compute the \emph{annealed} version of such quantities. An useful way of computing the latter is to rewrite such expectations as a non-Markovian process that simultaneously let the walks move and also explores the graph. More precisely, for every $B\subset V$, $\mu \in \mathcal{P}(V)$ and $t\geq 0$ it holds that
\begin{equation} \label{eq: annealing}
    \E[P_\mu(X_t\in B)] = \P^{\rm an}_\mu (W_t \in B)\,,
\end{equation}
where $\P^{\rm an}_\mu$ is the law of the joint random variable describing the partial realisation of the DCM, as described in Section \ref{sec:DCM}, with at most $t$ matchings and $W_t\in V$ is the location of the annealed walk at time $t$ having $\mu$ as initial distribution. Such joint process is non-Markovian as at each time step $s$ the transition probabilities depend on the whole trajectory of explored vertices up to step $s$. The evolution can be described as follows: in an environment given by the empty matching of the edges, the walk samples its initial position $x\in[n]$ according to $\mu$; then it samples u.a.r. a tail of $x$ and u.a.r. a head $f$ incident to some vertex $y\in[n]$ among all the possible ones. The edge $(x,y)$ is formed and the walk moves from $x$ to $y$. The procedure iterates up to time $t$, where at each step, if the walk chooses a head that is already matched to a tail then no other edges are formed.

The $\ell$-th moment with respect to $\E$ of the transition probabilities associated to $P_\mu$ can be computed via a similar construction, using multiple random walks, for which the above construction reads as follows. For all $\ell\in\mathbb{N}$, $B_1,\dots,B_\ell \subset V$ and $t_1, \cdots,t_\ell \geq 0$ we can write
\begin{equation} \label{eq:annealing multiple}
    \E\Big[\prod_{i=1}^\ell P_\mu (X_{t_i} \in B_i)\Big] = \P^{\ell-\rm an}_\mu \Big(W^{(i)}_{t_i}\in B_i,\; \forall i\leq \ell\Big)\,,
\end{equation}
where the random variables $(W^{(i)})_{i\leq \ell}$ are independent annealed walks sampled analogously to the single annealed walk, with the difference that the $\ell$ walks are sampled sequentially, all having initial distribution $\mu$ and, for any $i<\ell$, the $i$-th walk evolves in the partial environment described by the first $i-1$ walks.

Another natural application of the annealing random walks is the computation of an expectation conditioned on a partial realisation of the environment $\gamma$, given by any partial matching of tails and heads according to the DCM uniform matching procedure. Let $\mu\in\mathcal{P}([n])$ depending only on $\gamma$, $\ell \in \mathbb{N}$, $B_1,\dots,B_\ell \subset V$ and $t_1, \cdots,t_\ell \geq 0$, then
\begin{equation} \label{eq:annealing conditioned}
    \E\Big[\prod_{i=1}^\ell P_\mu (X_{t_i} \in B_i)\mid \gamma\Big] = \P^{\ell-{\rm an} \mid \gamma}_\mu \Big(W^{(i)}_{t_i}\in B_i,\; \forall i\leq \ell\Big)\,,
\end{equation}
where $ \P^{\ell-{\rm an} \mid \gamma}_\mu $ is the joint law of the partial environment and the multiple random walks as described previously, but the initial environment of the first walk is given by $\gamma$ instead of the empty matching of the edges. Clearly, all the events in \eqref{eq: annealing} - \eqref{eq:annealing conditioned} can be replaced by any event $(\mathcal{A}_i)_{i\leq \ell}$ that depend only on the trajectories of $(X^{(i)}_{s})_{s\leq t_i, i\leq \ell}$.

 \section{Proof of the main result} \label{sec: Proof of the main result}

 This section is fully devoted to show the proof of Theorem \ref{thm:main}. It is divided into two subsections: in the first one we provide a proof for the main result only for short time scales, while in the second one we extend it to any time scale.
 
\subsection{Short time scales}

Fix $x,y\in[n]$ such that $x\to y$, and consider two discrete-time, asynchronous, independent random walks $X,Y$ on $G\times G$ such that $(X_0,Y_0)=(x,y)$. Let us introduce the following random variables that will become useful for the rest of the proof. 

Let $\mathcal{B}^+_x(h)$ be the out-neighbour of $x\in[n]$ in $G$ up to depth $h>0$, $$V^+_\star = \{x\in[n] \mid \mathcal{B}^+_x(\hbar) \mbox{ is a tree} \}\,,$$ and 
\begin{equation}
    \hbar = \hbar_n = \frac{1}{10} \frac{\log(n)}{\log(d^+_{\max})}\,.
\end{equation}
Define
\begin{align}
    \bar \tau &= \bar \tau^{(x,y)}=\inf\{t>0\mid X_t\notin (Y_s)_{s\leq t}\cup \{x\}\}\,, \label{eq:tau bar}\\
    \tau_{\rm dev} &= \tau_{\rm dev}^{(x,y)} = \inf\{t>0\mid \mathcal{B}^+_{X_t}(\hbar)\cap \mathcal{B}^+_{Y_t}(\hbar) = \emptyset \mbox{ and }X_t,Y_t\in V^+_\star \}\,.\label{eq:tau dev}
\end{align}
In words, $\tau_{\rm dev}$ is the first time such that the out-neighbour of the two walks are non-intersecting trees. As we will prove, this happens exactly when the walk $X$ does not follow the path of the walk $Y$. Moreover, let 
\begin{equation} \label{eq:nu dev}
    \nu_{\rm dev} (u,v) = \mathbf{P}(X_{\tau_{\rm dev}}=u, Y_{\tau_{\rm dev}}=v \mid (X_0,Y_0)=(x,y)), \quad u,v\in[n], u\neq v\,.
\end{equation}
The following lemma shows that the first time in which the walk starting at $x$ moves into a vertex that has not been visited previously by the walk starting at $y$, will happen in a short amount of time with high probability.

\begin{lemma} \label{lemma:tau-bar small}
    Suppose that the degree sequence satisfies Assumption \ref{deg-assumptionn}. Define the sequence $h_\star=h_\star^{(n)}$ such that $h_\star=\log^2(n)$. It holds that
    \begin{equation}
        \max_{x,y\in[n]} P(\bar\tau > h_\star)\mathds{1}_{x\to y} = o_\P(n^{-C\,\log(n)})\,,
    \end{equation}
    for some $C>0$.
\end{lemma}

\begin{proof}
    Fix $x,y\in[n]$ such that $x\to y$. Let $S(t)$ be the number of steps of the walk $X$ within time $t>0$. Since the walks are discrete-time and moves asynchronously, we have that $S(t)\overset{d}{=}\text{Bin}(t,1/2)$. Therefore,
    \begin{equation*}
    \begin{split}
        P(\bar\tau > h_\star) &= P(\bar\tau > h_\star,\, S(h_\star)\geq h_\star/3) + P(\bar\tau > h_\star,\,S(h_\star)< h_\star/3) \\
        &\leq (d^+_{\min})^{-\frac{h_\star}{3}} + \text{Pr} (\text{Bin}(h_\star,1/2)< h_\star/3) = o(n^{-\Bar{c}\,\log(n)})\,,
    \end{split}
    \end{equation*}
     for some $\Bar{c}>0$, where the first bound comes from the fact that in order to $\bar \tau$ not to happen within time $h_\star$ the walk $X$ needs to follow the path of $Y$ for all its steps within $h_\star$. We conclude using the definition of $h_\star$, the fact that $d^+_{\rm min} \geq 2$, by Assumption \ref{deg-assumptionn}, and that $\text{Pr} (\text{Bin}(h_\star,1/2)< h_\star/3) =o(n^{-c\,\log(n)})$, for some $c>0$, by Hoeffding's inequality.
\end{proof}


In the next lemma, we will prove that conditioned on the paths of the walks up to time $\bar \tau$, the out-neighbourhood generated by the positions of the walks at $\bar \tau$ gives two non intersecting trees of logarithmic size. Such result strongly relates the hitting times $\bar \tau$ and $\tau_{\rm dev}$.


\begin{lemma} \label{lemma:tau_dev_conditions}
Suppose that the degree sequence satisfies Assumption \ref{deg-assumptionn}. Let $\Bar{\tau}$ as in \eqref{eq:tau bar} and $\Xi = (X_s,Y_s)_{s\leq \bar \tau}\,$. For all $i\in\{1,2,3\}$, it follows that
\begin{equation*}
    \P^{{\rm an}\mid \Xi} ( \mathcal{A}_i ) = 1- o(1)\,,
\end{equation*}
where
\begin{equation*}
    \mathcal{A}_1= \{\mathcal{B}^+_{X_{\bar \tau}}(\hbar)\cap \mathcal{B}^+_{Y_{\bar \tau}}(\hbar) = \emptyset\},\quad \mathcal{A}_2 = \{X_{\bar \tau}\in V^+_\star\},\quad \mathcal{A}_3 = \{Y_{\bar \tau}\in V^+_\star\}\,,
\end{equation*}
and $\P^{{\rm an}\mid \Xi}$ denotes the conditioned annealing law, as defined in \eqref{eq:annealing conditioned}.
\end{lemma}

\begin{proof}
    Define a coupling between the exploration process of $\mathcal{B}^+_{X_{\bar \tau}}(\hbar)$ and an unimodular Galton-Watson tree, as defined in Section \ref{sec: Models and results}, rooted at $X_{\bar \tau}$ on a partial environment of $G$ given by the vertices and edges explored by $(X_s,Y_s)_{s\leq \bar \tau}$. Call $\hat{\P}$ the law of such a coupling. Let $\mathcal{F}_i$ be the event in which $\mathcal{A}_i$ fails, $i\in\{1,2,3\}$, with respect to $\hat \P$. Consider the uniform matching between tails and heads of the DCM. In order to analyze $\mathcal{A}_2$ and $\mathcal{A}_3$, let us construct sequentially the out-neighbour of $X_{\bar \tau}$ and $Y_{\bar \tau}$ up to $\hbar$ in the partial environment previously stated. Let $\sigma$ be the first time such that a head incident to a previously selected vertex is chosen and $h_\star =\log^2(n)$. Then, for any $t>0$,
    \begin{equation*}
        \hat\P(\sigma =t) \leq \hat\P(\sigma =t, \Bar{\tau}\leq h_\star) + \hat{\P}(\Bar{\tau}> h_\star)  \leq \frac{d^-_{\rm max} (t+h_\star)}{m} + o(n^{-C\,\log(n)})\,,
    \end{equation*}
    $C>0$, where during the $t$-th step in the matching procedure, there are at most $t+\bar \tau$ already matched vertices and, thanks to Lemma \ref{lemma:tau-bar small}, $\bar \tau \leq h_\star$ with high probability. Therefore
    \begin{equation*}
        \hat\P(\sigma \leq t) \leq \frac{t\,d^-_{\rm max} (t+h_\star)}{m} + t\, o(n^{-C\,\log(n)}) \,.
    \end{equation*}
    In order to explore the whole $\mathcal{B}^+_{X_{\bar \tau}}(\hbar)$ it is sufficient to take $t=(d^+_{\rm max})^\hbar$, thus 
    \begin{equation}\label{eq:failureF_1}
        \hat\P(\mathcal{F}_2) \leq \hat\P(\sigma \leq (d^+_{\rm max})^\hbar) \leq 2\frac{(d^+_{\rm max})^{2\hbar}\,d^-_{\rm max}}{m} + o(1)\,.
    \end{equation}
    The result follows by the definition of $\hbar$, and the fact that $d^-_{\rm max}$ and $d^+_{\rm max}$ are uniformly bounded by Assumption \ref{deg-assumptionn}. Similarly, the same conclusion follows for $\mathcal{F}_1$ and $\mathcal{F}_3$. In particular, the bound in \eqref{eq:failureF_1} reads
    \begin{equation*}
        \hat\P(\mathcal{F}_1)\leq \frac{(d^+_{\rm max})^{\hbar}\,d^-_{\rm max}(2(d^+_{\rm max})^{\hbar} + h_\star)}{m} + o(1) = o(1)\,.
    \end{equation*}
\end{proof}
\begin{corollary} \label{coro:tau_bar=tau_dev}
    Suppose that the degree sequence satisfies Assumption \ref{deg-assumptionn}. It holds that
    \begin{equation*}
       \max_{x, y\in[n]} P(\bar\tau^{(x,y)}=\tau^{(x,y)}_{\rm dev}) = 1-o_\P(1)\,.
    \end{equation*}
\end{corollary}
\begin{proof}
    The conclusion follows from the boundness of the random measure and fact that 
    \begin{equation*}
       \max_{x, y\in[n]} \E[P(\bar\tau^{(x,y)}=\tau^{(x,y)}_{\rm dev})] = 1-o(1)\,.
    \end{equation*}
    Indeed, $\bar\tau^{(x,y)}\leq\tau^{(x,y)}_{\rm dev}$, as until the walk $X$ does not take a different path with respect to the vertices explored by $Y$, the out-neighbour of $X$ and $Y$ will have a non-trivial intersection. The fact that $\bar\tau^{(x,y)}$ is the first time in which the conditions of $\tau^{(x,y)}_{\rm dev}$ are satisfied w.h.p. follows directly from Lemma \ref{lemma:tau_dev_conditions}.
\end{proof}
Recall the definition of $\tau_{\rm meet}^{(x,y)}$ in \eqref{eq:tau meet}. The following proposition gives us important information about the behaviour of the meeting time after the deviation time, i.e. after $\tau_{\rm dev}$ happened. In particular, it tells us that after deviating the walk will not meet again after at least $\mathcal{O}(\log^3(n))$ steps with high probability. As shown in the proof of Corollary 8.6, \cite{ACHQ23}, it is related to the fact that first the walks have to exit the directed trees in which they are trapped for at least $\mathcal{O}(\log(n))$ amount of time, and after that they need at least another $\mathcal{O}(\log^2(n))$ steps to chase each other and finally meet. This fact will become useful once we notice that the mixing time of the product chain of the two independent walks has logarithmic order, thus w.h.p. after deviating the walks will first mix and then possibly meet.
\begin{proposition}[Corollary 8.6, \cite{ACHQ23}] \label{prop:meet_after_dev}
    Suppose that the degree sequence satisfies Assumption \ref{deg-assumptionn}. Then
     \begin{equation*}
        \max_{(u,v)\in \supp{\nu_{\rm dev}}} P(\tau^{(u,v)}_{\rm meet} > \log^3(n)) \overset{\P}{\longrightarrow} 0\,,
     \end{equation*}
    as $n\to\infty$, where $\supp{\nu_{\rm dev}}$ denotes the support of the measure $\nu_{\rm dev}$.
\end{proposition}

We deduce that, if we are interested in computing the first meeting time of two independent random walks within a short time scale, i.e. of the order of $t_n=o(\log^3(n))$, it suffices to consider only the events in which they do not deviate before they meet. More formally, the following corollary holds true.

\begin{corollary}\label{coro:dev after meet}
     Suppose that the degree sequence satisfies Assumption \ref{deg-assumptionn}. For any $x,y\in[n]$, $x\to y$, and any sequence $t=t_n$ such that $\lim_{n\to \infty}t_n$ exists and $t_n=o(\log^3(n))$, it holds that
     \begin{equation*}
         P(\tau_{\rm meet}^{(x,y)}\leq t) = P(\tau_{\rm meet}^{(x,y)}\leq t,\, \tau_{\rm dev}^{(x,y)}>\tau_{\rm meet}^{(x,y)}) + o_\P(1)\, 
     \end{equation*}
\end{corollary}

\begin{proof}
    It follows from Proposition \ref{prop:meet_after_dev} and the fact that $t_n=o(\log^3(n))$.
\end{proof}

The next result will be the building block for the asymptotic behaviour of the annealed expected number of discordant edges.

\begin{lemma} \label{lemma:annealed meeting}
Suppose that the degree sequence satisfies Assumption \ref{deg-assumptionn}. Fix $x,y\in[n]$, $x\neq y$, such that $x\to y$, and any sequence $t=t_n$ such that $\lim_{n\to \infty}t_n$ exists and $t_n=o(\log^3(n))$. For any $s\leq \frac{t}{2}$ and $\varepsilon\in(0,1)$, it holds that
\begin{equation}\label{eq:annealed meeting 1}
     \P^{{\rm an} \mid(x,y)} (\tau_{\rm meet}^{(x,y)}=2s,\, \tau_{\rm dev}^{(x,y)}>\tau_{\rm meet}^{(x,y)}) = o(n^{\varepsilon-1})\,,
\end{equation}
and
\begin{equation}\label{eq:annealed meeting 2}
     \P^{{\rm an} \mid(x,y)} (\tau_{\rm meet}^{(x,y)}=2s+1,\, \tau_{\rm dev}^{(x,y)}>\tau_{\rm meet}^{(x,y)}) = 2^{-1}\frac{1}{d^+_x} \mathds{1}_{s=0} + 2^{-2s-1}C_s\frac{1}{d^+_x}\frac{1}{d^+_y} \rho^{s-1}\mathds{1}_{s>0} + o(n^{\varepsilon-1})\,,
\end{equation}
where $C_s$ are the Catalan numbers, $\rho$ as in \eqref{eq:def-rho-gamma-delta-beta}  and $\P^{{\rm an} \mid(x,y)}$ is the conditioned annealed law described in \eqref{eq:annealing conditioned}.
\end{lemma}
\newcommand{\Txy}{\mathcal{T}^{(x,y)}}
\newcommand{\T}{\mathcal{T}}
\begin{proof}
Similar to what we described in Section \ref{sec:DCM}, consider the coupling between the exploration process generated by the annealed random walks on the original graph, conditioned on $x\to y$, having law $\P^{{\rm an} \mid(x,y)}$, and the same annealed process defined on the out-directed Galton-Watson tree rooted at $x$, still conditioned on $x\to y$, with offspring distribution given by
\begin{equation} \label{eq:mu+}
    \mu^+(k) = \sum_{z\in[n]} \frac{d_z^-}{m}\mathds{1}_{d_z^+=k}, \quad k\in \mathbb{Z}_+\,,
\end{equation}
and call it $\Txy$. Denote with $\hat{\P}$ the law of such a coupling, and with $\mathcal{F}$ the event that the coupling fails within time $t$. Analogous to what we established in Lemma \ref{lemma:tau_dev_conditions}, we can give an upper bound on the latter event as follows
\begin{equation} \label{eq:failing_coupling_tree}
    \hat{\P}(\mathcal{F})\leq t^2 \frac{d^-_{\rm max}}{m} = o(n^{\varepsilon-1})\,,
\end{equation}
for any $\varepsilon\in(0,1)$, thanks to Assumption \ref{deg-assumptionn} on the maximal in-degree and the fact that $t_n=o(\log^3(n))$.

It follows that we can consider the annealed process of interest only on the GW tree structure. Therefore, under the event $\tau_{\rm dev}^{(x,y)}>\tau_{\rm meet}^{(x,y)}$, the only possibility for the two walks starting at $x,y$ to meet is to follow each other on the same branch of the random tree rooted at $x$ and eventually meet. Due to parity conditions, since the walks start at distance one in $(X_0,Y_0)=(x,y)$, for any $s\leq \frac{t}{2}$ we can rule out the even times and thus we have that 
\begin{equation*}
    \P^{{\rm an} \mid(x,y)} (\tau_{\rm meet}^{(x,y)}=2s,\, \tau_{\rm dev}^{(x,y)}>\tau_{\rm meet}^{(x,y)}) = o(n^{\varepsilon-1})\,,
\end{equation*}
as expected, proving therefore the first part of the lemma. 

Let us consider the case in which the meeting happens at odd times on the event that the deviation time did not happen before the first meeting. If $s=0$, i.e. $\tau_{\rm meet} =1$, the event occurs if and only if the walk $X$ is selected and it moves in the direction of $Y$, to vertex $y$. Therefore 
\begin{equation}
     \P^{{\rm an} \mid(x,y)} (\tau_{\rm meet}^{(x,y)}=1,\, \tau_{\rm dev}^{(x,y)}>\tau_{\rm meet}^{(x,y)}) = \frac{1}{2}\frac{1}{d^+_x} + o(1)\,.
\end{equation}
 Fix $1\leq s\leq \frac{t}{2}$.  Thanks to \eqref{eq:failing_coupling_tree} it is enough to consider the same probabilities under the event that the coupling with the tree succeeds, therefore we restrict the focus on the construction of the G-W tree $\Txy$ with law $\hat\P$. Let $d(s)$ be the distance in $\Txy$ between $X_s$ and $Y_s$ at time $s$. We can rewrite the event in \eqref{eq:annealed meeting 2} in terms of the Markov process induced by $d(s)$. Recall the definition of $\Bar{\tau}=\Bar{\tau}^{(x,y)}$ in \eqref{eq:tau bar} and note that, on $\Txy$, it holds that
\begin{equation*}
    d(\Bar{\tau}) = d(t) = \infty, \quad \forall t\geq \Bar{\tau}\,.
\end{equation*}
Moreover, $(X_0,Y_0)=(x,y)$ implies that $d(0)=1$ and, given $d(s)$, we have that 
\begin{equation}
    d(s+1) = 
    \begin{cases}
        d(s) -1 & \text{ if } s < \Bar{\tau} \text{ and } X \text{ moves}\,, \\
        d(s) + 1 & \text{ if }  s < \Bar{\tau} \text{ and } Y \text{ moves}\,,\\
    \infty & \text{ if } s\geq \Bar{\tau}\,.
    \end{cases}
\end{equation}
Thanks to Corollary \ref{coro:tau_bar=tau_dev}, the events of the type $\{\tau_{\rm meet}^{(x,y)}=2s+1\} \cap \{\tau_{\rm dev}^{(x,y)}>\tau_{\rm meet}^{(x,y)}\}$ can be rewritten as
\begin{equation}
        \cC_s=\{0<d(r)<\infty\ , \forall r\in\{1,\dots,2s\} \}\bigcap \{d(2s +1)=0\}\,,\qquad s\le \frac{t}2\,.
    \end{equation}
We can look at $\cC_s$ as a collection of simple events of the type $\{d(1),\dots, d(2s+1)\}$ with proper constrains. The latter add up to an evolution in which the particle $X$ follows $Y$ one up to reaching it for the first time at time $2s+1$. By construction, they correspond to all the possible Dyck paths with $s$ upstrokes and $s+1$ downstrokes having $\pm1$ increments, starting from $d(0)=1$; call them $\mathfrak{D}_s$. Notice now that $\Txy$ can be seen as a rooted (at $x$), marked out-directed G-W tree, where each vertex $v\neq x,y$ has mark $\ell\in[n]$ with probability $\frac{d^-_z}{m}$, and $v$ has mark $\ell$ if and only if $d_v^+=d_\ell^+$. Let $\mathfrak{L}_s=(\ell_0,\dots,\ell_{s})\in [n]^s$ be the sequence of random marks of all the vertices in the path of length $s$ defined by the two random walks from the root to the vertex at which they meet (not included).

As a consequence, each simple event in $\cC_s$ can be associated uniquely to a couple $(\mathfrak{D}_s,\mathfrak{L}_s)$.

Observe that once we fix the marks of each vertex in which the walks move in their path of length $s$, the different events that contribute to a simple event $\{d(1),\dots, d(2s+1)\}\in\cC_s$ differ only in the order in which the particles moves. Thus they are all equiprobable as the walks are asynchronous and at each step the moving one is selected w.p. $1/2$. Moreover, it is known that the number of Dyck paths of length $2s+1$, $s\geq1$, is the Catalan number $C_s$. Therefore, it is enough to take a representative Dyck path $\mathfrak{D}_s$ and write
\begin{equation} \label{eq:expression1}
\begin{split}
    \hat\P(\tau_{\rm meet}^{(x,y)}=2s+1,\, \tau_{\rm dev}^{(x,y)}>\tau_{\rm meet}^{(x,y)})&=C_s\sum_{\mathfrak{L}_s\in[n]^s}\hat{\P}(\mathfrak{D}_s,\ \mathfrak{L}_s)\\
    &=
    2^{-2s-1}C_s\sum_{\mathfrak{L}_s\in[n]^s}\hat{\P}( \mathfrak{L}_s\mid \mathfrak{D}_s)\ .
\end{split}
\end{equation}
Recall the definition of $\mu^+$ in \eqref{eq:mu+} and consider the collection of independent random variables $D^+_0,D^+_1,\dots,D^+_{s}$ where $D^+_i\sim\mu^+$ for $i\ge 2$, while $D^+_0 =d_x^+$ and $D^+_1 =d_y^+$ almost surely.
Then
\begin{equation}\label{eq:expression2}
\begin{split}
    \sum_{\mathfrak{L}_s\in[n]^s}\hat{\P}( \mathfrak{L}_s\mid \mathfrak{D}_s)
    &= \sum_{d_0\geq 2}\dots \sum_{d_{s}\geq 2}\; \prod_{j=0}^{s} \frac{1}{d_j} \,\hat\P(D_0^+=d_0,\dots,D^+_s=d_{s}) \\
    &= \sum_{d_0,d_1,\dots,d_{s} \geq 2} \;\prod_{j=0}^{s}\frac{1}{d_j}\,\hat\P(D^+_j=d_j)\\
    &= \frac{1}{d_x^+}\frac{1}{d_y^+}\; \left(\sum_{k\geq 2} \frac{1}{k}\,\mu^+(k)\right)^{s-1}= \frac{1}{d_x^+}\frac{1}{d_y^+}\;\rho^{s-1}\,,
\end{split}
\end{equation}
where the second equality comes from the independence of the samplings of $D_i^+$, and the last one is a consequence of the following identity
\begin{equation}
   \sum_{k\geq 2} \frac{1}{k}\,\mu^+(k) = \sum_{k\geq 2} \frac{1}{k}\,\sum_{x\in [n]} \frac{d^-_x}{m}\mathds{1}_{d^+_x =k} = \frac1m\sum_x \frac{d_x^-}{d_x^+}=\rho\ .
\end{equation}
We conclude by plugging \eqref{eq:expression2} into \eqref{eq:expression1}.

\end{proof}

\begin{proposition}[Expectation short time scales] \label{prop:first moment short time scales}
    Suppose that the degree sequence satisfies Assumption \ref{deg-assumptionn}. Fix $u\in(0,1)$ and let $\eta_0=\text{Bern}(u)^{\otimes V}$. Then, for any non-negative sequence $t_n$ such that $\lim_{n\to \infty}t_n$ exists and $t_n=o(\log^3(n))$, it holds that 
\begin{equation}
\bigg|\E\big[\mathbf{E}_u[\mathcal{D}_{t_n}]\big] - 2u(1-u)\,\varphi(t_n)\bigg| \underset{n\to \infty}{\longrightarrow} 0\,,
\end{equation}
where $\varphi(\cdot)$ is as in \eqref{eq:phi}. In particular, if $\lim_{n\to \infty}t_n = \infty$ it holds that $|\varphi(t_n)-\varphi(\infty)|\longrightarrow0$, where
\begin{equation}\label{eq:phi_infinity}
    \varphi(\infty) = 1-\frac{1-\sqrt{1-\rho}}{\delta\,\rho}\,.
\end{equation}
\end{proposition}

\begin{proof}

 First observe that the expected density of discordant edges can be rewritten as follows
\begin{equation} \label{expected_disc_edg}
\mathbf{E}_u[\mathcal{D}_{t_n}] = \frac{1}{m}\sum_{e\in E} \mathbf{P}_u(e\in D_{t_n}) = \frac{1}{m}\sum_{x,y\in[n]} \mathbf{P}_u((x,y)\in D_{t_n}) \,\mathds{1}_{(x,y)\in E}\,.
\end{equation}
By the classical duality between the voter model and a system of coalescing random walks, see Section \ref{sec:voter model}, we can deduce that the event that an edge $e$ is discordant at time $t$ can be expressed as the event that two independent random walks starting at vertices with discordant opinion do not meet within time $t$. In other words, for any $e=(x,y)\in E$ and $t\geq0$ it holds that 
$$
\mathbf{P}_u(e=(x,y)\in D_{t}) = 2u(u-1) \mathbf{P}(\tau_{\rm meet}^{(x,y)}>t)\,,
$$
where $\tau_{\rm meet}^{(x,y)}$ was defined in \eqref{eq:tau meet}. We will prove the result for the discrete-time, asynchronous embedded chain and after that, by a Poissonization argument, pass to the continuous-time version. Recall that, as shown in Section \ref{sec: Random walks on sparse DCM}, we denote by $P,E$ the law of the discrete-time walks on $G$. Plugging in the latter into \eqref{expected_disc_edg} leads to
\begin{equation*}
    E_u[\mathcal{D}_{t_n}] = \frac{2u(1-u)}{m}\sum_{x,y\in[n]} P(\tau_{\rm meet}^{(x,y)}>t)\,\mathds{1}_{(x,y)\in E}\,.
\end{equation*}
We need to analyse $E_u[\mathcal{D}_{t_n}]$ as a random variable with respect to the graph $G$, that is, w.r.t. $\P$. Therefore, we begin by studying
\begin{equation} \label{eq:discordant1}
    \begin{split}
    \E[ E_u[\mathcal{D}_{t_n}]] &= \frac{2u(1-u)}{m}\sum_{x,y\in[n]} \E[P(\tau_{\rm meet}^{(x,y)}>t)\,\mathds{1}_{(x,y)\in E}] \\
    &= \frac{2u(1-u)}{m}\sum_{x,y\in[n]} \E[\mathds{1}_{(x,y)\in E}\E[P(\tau_{\rm meet}^{(x,y)}>t)\mid (x,y)]] \,,
    \end{split}
\end{equation}
so that we can analyse the following quantity
\begin{equation*}
    \P^{{\rm an} \mid(x,y)} (\tau_{\rm meet}^{(x,y)}>t) = \E[P(\tau_{\rm meet}^{(x,y)}>t)\mid (x,y)]\,.
\end{equation*}
Here $\P^{{\rm an} \mid(x,y)}$ denotes the usual \emph{annealed law}, with the difference that the initial environment is not the empty matching of the heads and tails, but the partial realization of the environment given by the directed edge $(x,y)$. We call $\P^{{\rm an} \mid(x,y)}$ the conditioned annealed law, as described in Section \ref{sec: Random walks on sparse DCM}.

Moreover, the probability that two fixed vertices $x,y\in[n]$ are connected via an edge $x\to y$ is proportional to their in- and out-degrees, i.e.
\begin{equation}
    \P(x\to y) = \frac{d_x^+ \,d_y^-}{m-o(m)}\,, \ \text{ as } n\to \infty\,.
\end{equation}
Therefore, we can combine \eqref{eq:discordant1} together with Lemma \ref{lemma:annealed meeting} and Corollary \ref{coro:dev after meet} in order to get
\begin{equation} \label{eq:expression first moment discrete time}
\begin{split}
     \E[ E_u[\mathcal{D}_{t_n}]] &= \frac{2u(1-u)}{m}\sum_{\substack{x\in[n] \\ y\neq x}} \frac{d_x^+ \,d_y^-}{m-o(m)} \Bigg(1-  \P^{{\rm an} \mid(x,y)} (\tau_{\rm meet}^{(x,y)}\leq t, \tau^{(x,y)}_{\rm dev} > \tau_{\rm meet}^{(x,y)}) + o(1)\Bigg) \\
     &= \frac{2u(1-u)}{m}\sum_{\substack{x\in[n] \\ y\neq x}} \frac{d_x^+ \,d_y^-}{m-o(m)} \Bigg(1-  \P^{{\rm an} \mid(x,y)} (\tau_{\rm meet}^{(x,y)}\leq t, \tau^{(x,y)}_{\rm dev} > \tau_{\rm meet}^{(x,y)}) \Bigg) + o(1)\\
     &\overset{(\star)}{=} \frac{2u(1-u)}{m}\sum_{\substack{x,y\in[n]}} \frac{d_x^+ \,d_y^-}{m-o(m)}\bigg[ 1-  2^{-1}\frac{1}{d^+_x}\mathds{1}_{t>0} - \sum_{1\leq s\leq \floor{\frac{t-1}{2}}}2^{-2s-1}C_s\frac{1}{d^+_x}\frac{1}{d^+_y} \rho^{s-1}\,\mathds{1}_{t>2}\bigg] + o(1)\\
     &\sim 2u(1-u)\bigg[ 1 -\frac{1}{m^2}\sum_{x,y\in[n]} \frac{d_y^-}{2}\mathds{1}_{t>0} - \frac{1}{m^2}\sum_{x,y\in[n]} \frac{d_y^-}{d^+_y} \sum_{1\leq 
     s\leq \floor{\frac{t-1}{2}}} 2^{-2s-1}C_s \rho^{s-1}\,\mathds{1}_{t>2}  \bigg] \\ 
     &= 2u(1-u) \bigg[ 1 - \frac{1}{2\,\delta} \Bigg(\mathds{1}_{t>0} + \sum_{s=1}^{\floor{\frac{t-1}{2}}} 2^{-2s}\,C_s\,\rho^s \,\mathds{1}_{t>2}\Bigg)\bigg]\,,
\end{split}
\end{equation}

where we used that $m=\sum_{x\in [n]}d_x^+ = \sum_{x\in [n]}d_x^-$, and we denoted by $\delta=\delta_n=\frac{m}{n}$ the average in- and out-degree. The additive error $o(1)$ in the first line can be pulled out of the sum since the remaining factor term is of order $\Theta(1)$. As shown in \eqref{eq:failing_coupling_tree}, the additive $o(1)$ error term in the $(\star)$ equality comes from the failing probability of the coupling between the local exploration of the graph and the G-W tree $\Txy$ and it has an order of magnitude of the type $\frac{\log^\alpha(n)}{n}$, for some $\alpha>1$, thus the extra $t_n=o(\log^3(n))$ factor coming from the sum in the second equality will not affect it. Finally, in the $(\star)$ equality of \eqref{eq:expression first moment discrete time} we used the following fact
\begin{equation*}
    \frac{2u(1-u)}{m} \sum_{x\in[n]}  \frac{d_x^+ \,d_x^-}{m-o(m)}  \Bigg(1-  \P^{{\rm an} \mid(x,y)} (\tau_{\rm meet}^{(x,y)}\leq t, \tau^{(x,y)}_{\rm dev} > \tau_{\rm meet}^{(x,y)}) + o(1)\Bigg) \leq \frac{c\, d^+_{\rm max}\,d^-_{\rm max}}{n} \to 0\,,
\end{equation*}
for some $c>0$.

We retrieve the desired expression $2u(1-u) \varphi(t)$ after passing to the continuous-time setting. In fact
\begin{equation} \label{eq:continuous-time to discrete-time}
\begin{split}
      \E[ \mathbf{E}_u[\mathcal{D}_{t_n}]] &= \frac{2u(1-u)}{m}\sum_{x,y\in[n]} \E[\mathbf{P}(\tau_{\rm meet}^{(x,y)}>t)\,\mathds{1}_{(x,y)\in E}] \\
      &= \frac{2u(1-u)}{m} \sum_{k\geq 0} e^{-2t} \frac{(2t)^k}{k!}\Bigg( \sum_{x,y\in[n]} \E\Big[P\Big(\tau_{\rm meet}^{(x,y)}>\floor{\frac{k-1}{2}}\Big)\,\mathds{1}_{(x,y)\in E}\Big]\Bigg)\\
      &\overset{\eqref{eq:expression first moment discrete time}}{\sim} 2u(1-u) \varphi(t)\,.
\end{split}
\end{equation}

We conclude by noticing that \eqref{eq:phi_infinity} is a direct consequence of a manipulation in order to get the generating function of the Catalan numbers. Indeed
\begin{equation*}
 \sum_{s=1}^{\infty} 2^{-2s}\,C_s\,\rho^s = \sum_{s= 0}^\infty C_{s}\, \bigg(\frac{\rho}{4}\bigg)^{s} -1 = G(\rho/4) -1 =\frac{1-\sqrt{1-\rho}}{\rho/2} -1\,,
\end{equation*}
where $$ G(x) =\frac{1-\sqrt{1-4x}}{2\,x}, \quad x\in(0,1),$$ is the generating function of the Catalan numbers, (see, e.g., \cite[Ch. 5.4]{GKP89} or \cite[Eq. 24]{FL03}). Moreover, in the expression \eqref{eq:phi}, the elements with $k\to \infty$ become dominant as $t\to\infty$.
This concludes the proof. 
\end{proof}

\begin{proposition}[Concentration under short time scales] \label{prop:convergence in probability short time scales}
    Under the same assumptions of Proposition \ref{prop:first moment short time scales} it holds that
    \begin{equation}
\bigg|\mathbf{E}_u[\mathcal{D}_{t_n}] - 2u(1-u)\,\varphi(t_n)\bigg| \overset{\P}{\longrightarrow} 0\,.
    \end{equation}
\end{proposition}

\begin{proof}
    We have to show
    \begin{equation} \label{eq:concentration condition second moment}
        \E\big[\mathbf{E}_u[\mathcal{D}_{t_n}]^2\big] = (1+o(1))\,\E\big[\mathbf{E}_u[\mathcal{D}_{t_n}]\big]^2\,,
    \end{equation}
    so that the desired result then follows by Proposition \ref{prop:first moment short time scales} and Chebyshev's inequality. We start by rewriting
\begin{equation} \label{eq:second moment expr}
    \begin{split}
        \E\big[\mathbf{E}_u[\mathcal{D}_{t_n}]^2\big] &= \frac{1}{m^2}\sum_{e\in E}\sum_{e'\in E} \E\Big[\mathbf{P}_u(e\in D_{t_n}) \mathbf{P}_u(e'\in D_{t_n})\Big] \\
        &= \frac{1}{m^2}\sum_{x,y\in[n]}\sum_{x',y'\in[n]} \E\Big[\mathbf{P}_u((x,y)\in D_{t_n}) \mathbf{P}_u((x',y')\in D_{t_n}) \,\mathds{1}_{(x,y)\in E}\mathds{1}_{(x',y')\in E}\Big] \\
        &= \frac{(2u(1-u))^2}{m^2}\sum_{x,y\in[n]}\sum_{x',y'\in[n]} \E\Big[ \mathbf{P}(\tau_{\rm meet}^{(x,y)}>t_n) \mathbf{P}(\tau_{\rm meet}^{(x',y')}>t_n) \,\mathds{1}_{(x,y)\in E}\mathds{1}_{(x',y')\in E}\Big]\,,
    \end{split}
\end{equation}
where in the last equality we used the duality between voter model and coalescing random walks, as stated in the proof of Proposition \ref{prop:first moment short time scales}. Among all vertices $x,y,x',y'$ there are six different cases:
\begin{align*}
 1)\, x'\neq x,y \text{ and } y'\neq x,y\quad &4)\, x'\neq x,y \text{ and } y'=x  \\ 
 2)\, x'=x \text{ and } y'\neq x,y \qquad    &5)\, x'\neq x,y \text{ and } y'=y   \\  
 3)\, x'=y \text{ and } y'\neq x,y  \qquad     &6)\, x'=x \text{ and } y'=y        
\end{align*}
We claim that the only one giving a positive contribution to \eqref{eq:second moment expr} is 1), while all the other cases are vanishing terms. In fact, if we consider case 2), it holds that
\begin{align*}
    \frac{(2u(1-u))^2}{m^2} &\sum_{x,y\in[n]}\sum_{y'\in[n]\setminus\{x,y\}} \E\Big[ \mathbf{P}(\tau_{\rm meet}^{(x,y)}>t_n) \mathbf{P}(\tau_{\rm meet}^{(x,y')}>t_n) \,\mathds{1}_{(x,y)\in E}\mathds{1}_{(x,y')\in E}\Big] \\
    &\leq \frac{1}{m^2} \sum_{x,y,y'\in[n]} \frac{d_x^+\,d_y^-}{m} \frac{d_x^+\,d_{y'}^-}{m} = \frac{1}{m}\sum_{x\in [n]} \frac{(d_x^+)^2}{m} \leq \frac{C}{m}\longrightarrow 0\,,
\end{align*}
for some $C>0$ according to Assumption \ref{deg-assumptionn}. We can use an analogous argument to show that the contribution of cases 3)- 6) are vanishing. Therefore, we have to study \eqref{eq:second moment expr} with the quadruple of vertices of type 1), i.e.

\begin{equation} \label{eq:second moment 1)}
    \frac{(2u(1-u))^2}{m^2}\sum_{x,y\in[n]}\sum_{x',y'\neq x,y} \E\Big[ \mathbf{P}(\tau_{\rm meet}^{(x,y)}>t_n) \mathbf{P}(\tau_{\rm meet}^{(x',y')}>t_n) \,\mathds{1}_{(x,y)\in E}\mathds{1}_{(x',y')\in E}\Big]\,.
\end{equation}
In what follows we will pass to the discrete-time embedded chain, then moving back to the original continuous-time one by a Poissonization argument as shown in \eqref{eq:continuous-time to discrete-time}. Conditionally on $(x,y)$ and $(x',y')$, we are left to study
\begin{equation}\label{eq:annealed second moment}
    \P^{2-{\rm an} \mid \chi} (\tau_{\rm meet}^{(x,y)}\leq t_n,\, \tau_{\rm meet}^{(x',y')}\leq t_n)\,,
\end{equation}
where $\chi$ is the information containing both $(x,y)$ and $(x',y')$, while $ \P^{2-{\rm an} \mid \chi}$ represents the annealed law of two couples of independent random walks over the partial environment given by the empty matching of the edges conditioned on $\chi$. 

Similarly to the proof of Proposition \ref{prop:first moment short time scales}, we can \eqref{eq:second moment 1)}, and thus \eqref{eq:annealed second moment}, as follows
\begin{equation}
    \P^{2-{\rm an} \mid \chi} (\tau_{\rm meet}^{(x,y)}\leq t_n,\, \tau^{(x,y)}_{\rm dev} > \tau_{\rm meet}^{(x,y)}\,, \tau_{\rm meet}^{(x',y')}\leq t_n\,,\tau^{(x',y')}_{\rm dev} > \tau_{\rm meet}^{(x',y')})\,,
\end{equation}
up to a vanishing additive error term. As in the proof of Proposition \ref{prop:first moment short time scales}, we can immediately rule out the even times so that we are interested in
\begin{equation}
    \P^{2-{\rm an} \mid \chi} (\tau_{\rm meet}^{(x,y)} = 2s_1+1\,, \tau^{(x,y)}_{\rm dev} > \tau_{\rm meet}^{(x,y)}\,, \tau_{\rm meet}^{(x',y')}=2s_2+1\,,\tau^{(x',y')}_{\rm dev} > \tau_{\rm meet}^{(x',y')})\,,
\end{equation}
for any $0\leq s_1\leq s_2\leq t_n/2$. For $i\in\{1,2\}$, call $(X_{s}^{(i)},Y_{s}^{(i)})_{s\leq t}$ the path up to time $t=t_n$ of the two couples of annealed walks such that $(X_0^{(1)},Y_0^{(1)})=(x,y)$ and  $(X_0^{(2)},Y_0^{(2)})=(x',y')$, $x,y\neq x',y'$. The final two steps that are needed in order to conclude the proof are the following: first show that the path of the first two walks w.h.p. never intersects with vertices $x',y'$; then that the whole path of the second pair of walks $(X_{s}^{(2)},Y_{s}^{(2)})_{s\leq t}$ has w.h.p. empty intersection with $(X_{s}^{(1)},Y_{s}^{(1)})_{s\leq t}$.

Define 
\newcommand{\G}{\mathcal{G}}
\newcommand{\A}{\mathcal{A}}
\begin{equation}
     \G_s = \bigcup_{r\leq s} Y_r^{(1)}\,,\quad \text{and} \quad \A_s^{z,w} = \big\{z\in \G_s \cup w\in \G_s\big\},
\end{equation}
for any $z,w\in[n]$ and $s\geq 0$. First notice that, for all $s_1\leq t/2$, the following holds 
\begin{equation} \label{eq:x'y'xy}
    \P^{2-{\rm an} \mid \chi} (\A_{s_1}^{x',y'}) \leq \Bigg(\frac{d_{x'}^- + d_{y'}^-}{m-o(m)}\Bigg)^{s_1} \leq   \Bigg(C\,\frac{d_{\rm max}^-}{m}\Bigg)^{s_1} =o(1)\,,
\end{equation}
for some $C>0$, as $t_n=o(\log^3(n))$ and $d_{\rm max}^- =o(n)$ by Assumption \ref{deg-assumptionn}.
Then
\begin{align*}
      &\P^{2-{\rm an} \mid \chi} (\tau_{\rm meet}^{(x,y)} = 2s_1+1\,, \tau^{(x,y)}_{\rm dev} > \tau_{\rm meet}^{(x,y)}\,, \tau_{\rm meet}^{(x',y')}=2s_2+1\,,\tau^{(x',y')}_{\rm dev} > \tau_{\rm meet}^{(x',y')}) \\
      &= \P^{2-{\rm an} \mid \chi} (\tau_{\rm meet}^{(x,y)} = 2s_1+1\,, \tau^{(x,y)}_{\rm dev} > \tau_{\rm meet}^{(x,y)}\,, \tau_{\rm meet}^{(x',y')}=2s_2+1\,,\tau^{(x',y')}_{\rm dev} > \tau_{\rm meet}^{(x',y')}, (\A_{s_1}^{x',y'})^c)\\
      &+ o(1)\,.
\end{align*}
    Let $(\bar{X}_{s},\bar{Y}_{s})_{s\leq t}$ be two independent random walks on $G$, with $(\bar{X}_{0},\bar{Y}_{0})= (X_{0}^{(2)},Y_{0}^{(2)})$ and such that for all $s\leq t/2$
\begin{equation}
    \Bar{X}_{s}\notin \G_{s_1}\,.
\end{equation}
We conclude the proof by a coupling argument between $(X_{s}^{(2)},Y_{s}^{(2)})_{s\leq t}$ and $(\bar{X}_{s},\bar{Y}_{s})_{s\leq t}$, conditionally on $(X_{s}^{(1)},Y_{s}^{(1)})_{s\leq t}$. We let the two processes evolve independently until $X^{(2)}_s\in  \G_{s_1}$ for some $s\leq t$, then we reject the move and resample $X^{(2)}_s$. We say that the coupling fails at time $r$ if it is the first time such that a move is rejected.
Fix $s_2\leq t/2$. For any $s\leq s_2$, let $\mathcal{F}_s$ be the event in the coupled probability space $\hat{\P}^{2-{\rm an} \mid \chi}$ such that the above coupling fails at step $s$, and let $\mathcal{F}=\cup_{s\leq s_2}\mathcal{F}_s$. Then, conditionally on $\G_{s_1}$, we have
\begin{equation}
    \hat{\P}^{2-{\rm an} \mid \chi} (\mathcal{F}_s) \leq \frac{d^-_{\rm max}\, s_1}{m}
\end{equation}
for any $s\leq s_2$, since this corresponds to the probability of selecting a head of any of the vertices in $\G_{s_1}$. Thus by the union bound we get
\begin{equation}\label{eq:coupling second moment}
    \hat{\P}^{2-{\rm an} \mid \chi} (\mathcal{F}) \leq \frac{d^-_{\rm max}\, s_1\,s_2}{m} \leq \frac{d^-_{\rm max}\,\log^6(n)}{m}= o(1)\,.
\end{equation}
Note that in order to prove \eqref{eq:concentration condition second moment} it suffices to show $ \E\big[\mathbf{E}_u[\mathcal{D}_{t_n}]^2\big] \leq(1+o(1))\,\E\big[\mathbf{E}_u[\mathcal{D}_{t_n}]\big]^2$, as the reversed inequality is trivially satisfied.
  Moreover, notice that we can prove the latter w.l.o.g. in the usual discrete-time setting. To this aim, we can upper bound the discrete-time version of \eqref{eq:second moment 1)} as follows
\begin{equation}
\begin{split}
        &\frac{(2u(1-u))^2}{m^2}\sum_{x,y\in[n]}\sum_{x',y'\neq x,y} \frac{d_x^+\,d_y^-}{m}\frac{d_{x'}^+\,d_{y'}^-}{m}\Big[1 - o(1) \\
        &-\sum_{s_1,s_2\leq t} \P^{2-{\rm an} \mid \chi} (\tau_{\rm meet}^{(x,y)} = 2s_1+1\,, \tau^{(x,y)}_{\rm dev} > \tau_{\rm meet}^{(x,y)}\,, \tau_{\rm meet}^{(x',y')}=2s_2+1\,,\tau^{(x',y')}_{\rm dev} > \tau_{\rm meet}^{(x',y')}, (\A_{s_1}^{x',y'})^c, \cF^c)\Big] \\
        &=\frac{(2u(1-u))^2}{m^2}\sum_{x,y\in[n]} \frac{d_x^+\,d_y^-}{m}\bigg[ 1-  2^{-1}\frac{1}{d^+_x} - \sum_{1\leq s_1\leq t}2^{-2s_1-1}C_{s_1-1}\frac{1}{d^+_x}\frac{1}{d^+_y} \rho^{s_1-1}\bigg] \\
        &\times\sum_{x',y'\neq x,y}\frac{d_{x'}^+\,d_{y'}^-}{m} \bigg[ 1-  2^{-1}\frac{1}{d^+_{x'}} - \sum_{1\leq s_2\leq t}2^{-2s_2-1}C_{s_2-1}\frac{1}{d^+_{x'}}\frac{1}{d^+_{y'}} \rho^{s_2-1}\bigg]- o(1) \\
        &\leq\frac{(2u(1-u))^2}{m^2}\sum_{x,y\in[n]} \frac{d_x^+\,d_y^-}{m}\bigg[ 1-  2^{-1}\frac{1}{d^+_x} - \sum_{1\leq s_1\leq t}2^{-2s_1-1}C_{s_1-1}\frac{1}{d^+_x}\frac{1}{d^+_y} \rho^{s_1-1}\bigg] \\
        &\times\sum_{x',y'\in[n]}\frac{d_{x'}^+\,d_{y'}^-}{m} \bigg[ 1-  2^{-1}\frac{1}{d^+_{x'}} - \sum_{1\leq s_2\leq t}2^{-2s_2-1}C_{s_2-1}\frac{1}{d^+_{x'}}\frac{1}{d^+_{y'}} \rho^{s_2-1}\bigg]- o(1) \\
        &\leq(1+o(1))\,\E\big[E_u[\mathcal{D}_{t}]\big]^2\,,
\end{split}
\end{equation}
where in the first inequality we used \eqref{eq:x'y'xy}, \eqref{eq:coupling second moment} and Lemma \ref{lemma:annealed meeting}. We conclude the proof by noticing that for any $\varepsilon>0$ it holds that
\begin{equation}
    \P\Big(\big|\mathbf{E}_u[\mathcal{D}_{t}]- \E\big[\mathbf{E}_u[\mathcal{D}_{t}]\big]\big|>\varepsilon\Big) \leq \frac{\mathbb{V}\text{ar}(\mathbf{E}_u[\mathcal{D}_{t}])}{\varepsilon^2} \overset{\eqref{eq:concentration condition second moment}}{=} o(1)\,.
\end{equation}

\end{proof}

\subsection{Long time scales}
In this section we extend the result shown in Proposition \ref{prop:convergence in probability short time scales} to time scales $t=t_n$ that are of any order up to linear, i.e. such that $\lim_{n\to \infty}\frac{t_n}{n}=\ell$, for some $\ell>0$. The following is an adaptation of the results in \cite{ABHHQ22} to our directed and inhomogeneous framework. Recall the definitions of deviation time $\tau_{\rm dev}$ in \eqref{eq:tau dev} and $\nu_{\rm dev}$ in \eqref{eq:nu dev}. We start by proving a preliminary result that will be the key step bridging the convergence of short and long time scales. It shows that the tail distribution of the first meeting time of two independent walks starting right after deviating, thus w.h.p. far apart in the sense of Proposition \ref{prop:meet_after_dev}, has an exponential decay decreasing linearly in the size of the graph. 

\begin{lemma} \label{lemma:exponential meet from dev}
It holds that
    \begin{equation} \label{eq:exponential tail from dev}
       \max_{(u,v)\in \supp{\nu_{\rm dev}}} \sup_{t\geq 0} \big|\mathbf{P}(\tau^{(u,v)}_{\rm meet}>t) - e^{-2\frac{t}{n}\vartheta^{-1}}\big| \overset{\P}{\longrightarrow} 0\,,
    \end{equation}
    where $\vartheta$ is as in \eqref{theta}.
\end{lemma}

\begin{proof}
The result follows from a modification of Lemma 3.8 in \cite{ABHHQ22}, replacing $\tau_{\rm far}$ with $\tau_{\rm dev}$, where in their setting $\tau_{\rm far}$ is the first time such that two walks are at distance at least $(\log\log(n))^2$. For completeness we will sketch the proof in our framework. 
    Let $(u,v)\in \supp{\nu_{\rm dev}}$, then Proposition \ref{prop:meet_after_dev} implies that $\forall t \leq \log^2(n)$ 
\begin{equation}
    \mathbf{P}(\tau_{\rm meet}^{(u,v)}<t) \overset{\P}{\longrightarrow} 0, \quad \text{as } n\to \infty\,.
\end{equation}
In particular 
\begin{equation} \label{eq:meet from support dev}
    \min_{(u,v)\in \supp{\nu_{\rm dev}}}  \mathbf{P}(\tau_{\rm meet}^{(u,v)}>t) \geq 1 - o_\P(1)\,.
\end{equation}
\newcommand{\ttwo}{t_{\rm mix}^{\otimes2}}
\newcommand{\pitwo}{\pi^{\otimes2}}
Let $\ttwo$ be the mixing time and $\pitwo:=\pi\otimes\pi$ the stationary distribution of the product chain $(X,Y)$ given by two independent random walks on $G$. It holds that $\ttwo< \log^2(n)$ with high probability, see \cite{BCS18}, \cite{BCS19}. Thus it is sufficient to prove the result for $t>\ttwo$. Under such assumption, it holds that,
\begin{equation}\label{eq:meeting from dev eq1}
\begin{split}
 \mathbf{P}(\tau_{\rm meet}^{(u,v)}>t) &= \sum_{\substack{x,y\in [n]\\ x\neq y}} \mathbf{P}(\tau_{\rm meet}^{(u,v)}>t, X_{\ttwo} = x, Y_{\ttwo} =y, \tau_{\rm meet}^{(u,v)}>\ttwo)  \\
 &= \sum_{\substack{x,y\in [n]\\ x\neq y}} \mathbf{P}(\tau_{\rm meet}^{(u,v)}>\ttwo, X_{\ttwo} = x, Y_{\ttwo} =y) \mathbf{P}(\tau_{\rm meet}^{(u,v)}>t-\ttwo) \\
 &\geq \sum_{\substack{x,y\in [n]\\ x\neq y}} \mathbf{P}(X_{\ttwo} = x, Y_{\ttwo} =y) \mathbf{P}(\tau_{\rm meet}^{(u,v)}>t-\ttwo) - \mathbf{P}(\tau_{\rm meet}^{(u,v)}\leq \ttwo) \\
 &\geq \sum_{\substack{x,y\in [n]\\ x\neq y}} \pitwo(x,y) \mathbf{P}(\tau_{\rm meet}^{(u,v)}>t-\ttwo) - o_\P(1)\,,
\end{split}
\end{equation}
where the last inequality is a consequence of \eqref{eq:meet from support dev} and the fact that the law of $(X,Y)$ can be approximated by its stationary measure up to a vanishing error. We conclude by replacing the desired exponential term in \eqref{eq:meeting from dev eq1} thanks to Theorem \ref{thm:meeting from stationarity}. The upper bound follows by a similar argument.
\end{proof}

\begin{proposition} \label{prop:convergence in probability long time scales}
    Suppose that the degree sequence satisfies Assumption \ref{deg-assumptionn}. Fix $u\in(0,1)$ and let $\eta_0=\text{Bern}(u)^{\otimes V}$. Then, for any non-negative sequence $t_n$ such that $\lim_{n\to \infty}t_n=\infty$ and $\lim_{n\to\infty}\frac{t_n}{n}=\ell\geq 0$, it holds that 
\begin{equation}
\bigg|\mathbf{E}_u[\mathcal{D}_{t_n}] - 2u(1-u)\,\varphi(\infty)\, e^{-2\ell\,\vartheta^{-1}}\bigg| \overset{\P}{\longrightarrow} 0\,,
\end{equation}
where $\varphi(\cdot)$ is as in \eqref{eq:phi} and $\vartheta$ as in \eqref{theta}.
\end{proposition}
\begin{proof}
Let $h_\star=h_{\star,n}$ be a diverging sequence such that $h_\star=o(\log^3(n))$. It is enough to prove the result for $t=t_n>h_\star$, as for the complementary case it immediately follows by Proposition \ref{prop:convergence in probability short time scales}. Fix $x,y\in[n]$, and let
\begin{equation}
    \sigma = \sigma^{(x,y)} := \tau_{\rm meet}^{(x,y)} \wedge \tau_{\rm dev}^{(x,y)}\,.
\end{equation}
Recall  
\begin{equation} \label{eq:expected discordant edges1}
    \mathbf{E}_u[\mathcal{D}_{t_n}] =  \frac{2u(1-u)}{m}\sum_{x,y\in[n]} \mathbf{P}(\tau_{\rm meet}^{(x,y)}>t)\,\mathds{1}_{(x,y)\in E}\,,
\end{equation}
and, as a consequence of Lemma \ref{lemma:tau-bar small} and Corollary \ref{coro:tau_bar=tau_dev}, that
\begin{equation} \label{eq:tau dev small}
    \mathbf{P}(\tau_{\rm dev}^{(x,y)} > h_\star) \overset{\P}{\longrightarrow} 0\,.
\end{equation}
Therefore
\begin{equation}
    \mathbf{P}(\tau_{\rm meet}>t) =  \mathbf{P}(\tau_{\rm meet}>t,\, \sigma = \tau_{\rm dev}^{(x,y)},\, \tau_{\rm dev}^{(x,y)} \leq  h_\star) + o_\P(1)\,.
\end{equation}
It follows that
\begin{equation} \label{eq: equation 1}
\begin{split}
     &\mathbf{P}(\tau^{(x,y)}_{\rm meet}>t,\, \sigma = \tau_{\rm dev}^{(x,y)},\, \tau_{\rm dev}^{(x,y)} \leq  h_\star) \\
     &= \sum_{(u,v)\in \text{supp}(\nu_{\text{dev}})} \sum_{s\leq h_\star} \mathbf{P}(\tau^{(x,y)}_{\rm meet}>t\mid \sigma = \tau_{\rm dev}^{(x,y)},\, \tau_{\rm dev}^{(x,y)} = s,\, (X_\sigma,Y_\sigma)=(u,v)) \\
     &\times  \mathbf{P}(\sigma = \tau_{\rm dev}^{(x,y)},\, \tau_{\rm dev}^{(x,y)} =s,\, (X_\sigma,Y_\sigma)=(u,v))\\
     &= \sum_{(u,v)\in \text{supp}(\nu_{\rm dev})}  \sum_{s\leq h_\star}\mathbf{P}(\tau^{(u,v)}_{\rm meet}>t - s) \mathbf{P}(\sigma = \tau_{\rm dev}^{(x,y)},\, \tau_{\rm dev}^{(x,y)} = s,\, (X_\sigma,Y_\sigma)=(u,v))\\
      &=  e^{-2\frac{t}{n}\vartheta^{-1}}\sum_{(u,v)\in \text{supp}(\nu_{\rm dev})}  \sum_{s\leq h_\star}\mathbf{P}(\sigma = \tau_{\rm dev}^{(x,y)},\, \tau_{\rm dev}^{(x,y)} = s,\, (X_\sigma,Y_\sigma)=(u,v)) e^{-2\frac{s}{n}\vartheta^{-1}}\\
     &=  e^{-2\ell\vartheta^{-1}}\,
     \mathbf{P}(\sigma = \tau_{\rm dev}^{(x,y)},\, \tau_{\rm dev}^{(x,y)} \leq h_\star) + o_\P(1)\,
\end{split}
\end{equation}
as in the last equality we applied Lemma \ref{lemma:exponential meet from dev}. Finally
\begin{equation} \label{eq: equation 2}
\begin{split}
    \mathbf{P}(\sigma = \tau_{\rm dev}^{(x,y)},\, \tau_{\rm dev}^{(x,y)} \leq h_\star) &= \mathbf{P}(\tau_{\rm dev}^{(x,y)} < \tau_{\rm meet}^{(x,y)},\, \tau_{\rm dev}^{(x,y)} \leq h_\star) \\
    &= \mathbf{P}( \tau_{\rm dev}^{(x,y)} \leq h_\star ) -  \mathbf{P}(\tau_{\rm dev}^{(x,y)} \geq \tau_{\rm meet}^{(x,y)},\, \tau_{\rm dev}^{(x,y)} \leq h_\star) \\
    &\overset{\eqref{eq:tau dev small}}{=} 1- o_\P(1) -\mathbf{P}(\tau_{\rm dev}^{(x,y)} \geq \tau_{\rm meet}^{(x,y)},\, \tau_{\rm dev}^{(x,y)} \leq h_\star,\, \tau_{\rm meet}^{(x,y)}\leq h_\star) \\
    &= 1- o_\P(1) -\mathbf{P}(\tau_{\rm dev}^{(x,y)} \geq \tau_{\rm meet}^{(x,y)},\, \tau_{\rm meet}^{(x,y)}\leq h_\star) 
\end{split}
\end{equation}
If we now plug \eqref{eq: equation 1} and \eqref{eq: equation 2} into \eqref{eq:expected discordant edges1} we deduce that
\begin{equation}
\begin{split}
     \mathbf{E}_u[\mathcal{D}_{t_n}] &=  2u(1-u)\,e^{-2\frac{t}{n}\vartheta^{-1}}\,\frac{1}{m}\sum_{x,y\in[n]} \Big[1- o_\P(1) -\mathbf{P}(\tau_{\rm dev}^{(x,y)} \geq \tau_{\rm meet}^{(x,y)},\, \tau_{\rm meet}^{(x,y)}\leq h_\star) \Big]\,\mathds{1}_{(x,y)\in E}\\   
     &= 2u(1-u)\,\varphi(h_\star)\,e^{-2\ell\vartheta^{-1}} + o_\P(1)\,,
\end{split}
\end{equation}
where we exploited Proposition \ref{prop:convergence in probability short time scales} and the proof of Proposition \ref{prop:first moment short time scales}. We conclude by approximating $\varphi(h_\star)$ with $\varphi(\infty)$.
\end{proof}

\begin{proof}[Proof of Theorem \ref{thm:main}]
The result follows by Proposition \ref{prop:convergence in probability short time scales} for time scales $t_n=o(\log^3(n))$ while Proposition \ref{prop:convergence in probability long time scales} extends it to all the remaining ones. 
    
\end{proof}

\subsection*{Acknowledgements}
The work of F.C. is supported in part by the Netherlands Organisation for Scientific Research (NWO) through the Gravitation {\sc Networks} grant 024.002.003. The work of F.C. is further supported by the European Union's Horizon 2020 research and innovation programme under the Marie Sk\l odowska-Curie grant agreement no.\ 945045. The author thanks Luca Avena, Rajat Subhra Hazra, and Matteo Quattropani for their useful discussions and helpful suggestions.\hfill
\parbox{0.1\textwidth}
{~~~~\includegraphics[width=0.05\textwidth]{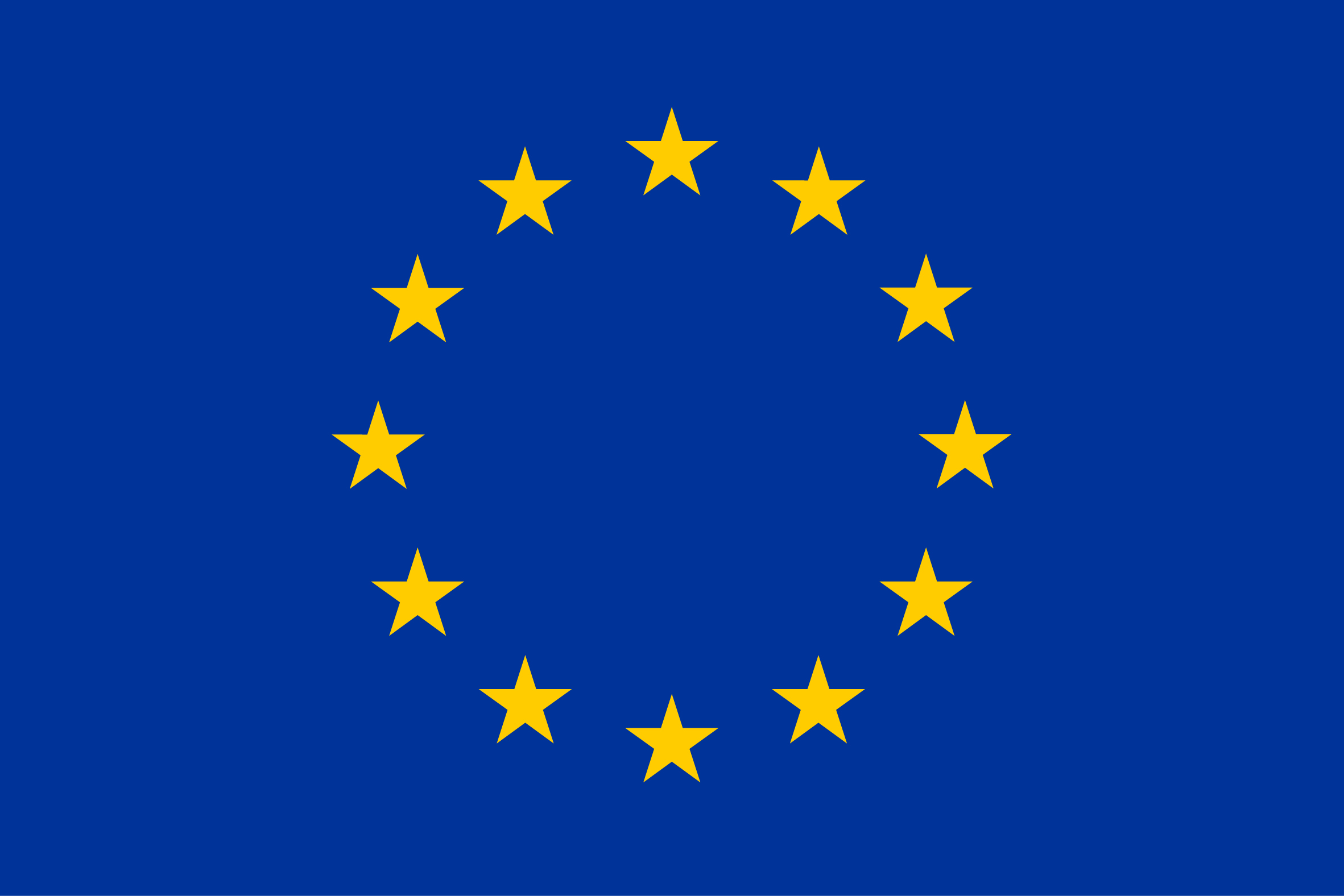}}

\bibliographystyle{abbrvnat}
\bibliography{reference.bib}
\end{document}